\documentclass[12pt,a4paper,final]{amsart}
\usepackage{amssymb}
\usepackage{hyperref}
\usepackage[notcite, notref]{showkeys}
\usepackage[utf8]{inputenc}

\theoremstyle{plain}
\newtheorem{theorem}{Theorem}[section]
\newtheorem{proposition}[theorem]{Proposition}
\newtheorem{corollary}[theorem]{Corollary}
\newtheorem{lemma}[theorem]{Lemma}

\newtheorem{condition}[theorem]{Condition}

\theoremstyle{definition}
\newtheorem{definition}[theorem]{Definition}
\newtheorem{example}[theorem]{Example}
\newtheorem{remark}[theorem]{Remark}
\newtheorem{question}[theorem]{Question}
\newtheorem{conjecture}[theorem]{Conjecture}

\newtheorem{notation}[theorem]{Notation}

\theoremstyle{remark}

\numberwithin{equation}{section}
\newcommand{\N}{\mathbb N}
\newcommand{\Z}{\mathbb Z}

\newcommand{\R}{\mathbb R}
\newcommand{\C}{\mathbb C}

\newcommand{\fa}{\mathfrak a}
\newcommand{\fb}{\mathfrak b}

\newcommand{\fg}{\mathfrak g}
\newcommand{\fh}{\mathfrak h}

\newcommand{\fk}{\mathfrak k}

\newcommand{\fp}{\mathfrak p}

\newcommand{\ft}{\mathfrak t}

\DeclareMathOperator{\GL}{GL}

\DeclareMathOperator{\Ot}{O}
\DeclareMathOperator{\SO}{SO}
\DeclareMathOperator{\SU}{SU}
\DeclareMathOperator{\Sp}{Sp}
\DeclareMathOperator{\Ut}{U}

\newcommand{\su}{\mathfrak{su}}

\newcommand{\op}{\operatorname}
\newcommand{\Id}{\textup{Id}}

\DeclareMathOperator{\Tr}{Tr}
\DeclareMathOperator{\diag}{diag}
\DeclareMathOperator{\Span}{Span}
\newcommand{\inner}[2]{\langle {#1},{#2}\rangle }
\newcommand{\innerdots}{\langle \cdot,\cdot \rangle }

\DeclareMathOperator{\Ad}{Ad}

\DeclareMathOperator{\Spec}{Spec}

\DeclareMathOperator{\diam}{diam}
\DeclareMathOperator{\vol}{vol}

\DeclareMathOperator{\dist}{dist}
\newcommand{\HH}{\mathcal H}
\newcommand{\CC}{\mathcal C}

\title[Diameter and Laplace eigenvalue estimates]{Diameter and Laplace eigenvalue estimates for left-invariant metrics on compact Lie groups}
\author{Emilio~A.~Lauret}
\address{Instituto de Matemática (INMABB), Departamento de Matemática, Universidad Nacional del Sur (UNS)-CONICET, Bahía Blanca, Argentina.}
\email{emilio.lauret@uns.edu.ar}

\thanks{This research was supported by grants from CONICET, FonCyT (BID-PICT 2018-02073), SeCyT, and the Alexander von Humboldt Foundation (return fellowship)}
\date{\today}

\begin{document}

\begin{abstract}
Let $G$ be a compact connected Lie group of dimension $m$.
Once a bi-invariant metric on $G$ is fixed, left-invariant metrics on $G$ are in correspondence with $m\times m$ positive definite symmetric matrices. 
We estimate the diameter and the smallest positive eigenvalue of the Laplace-Beltrami operator associated to a left-invariant metric on $G$ in terms of the eigenvalues of the corresponding positive definite symmetric matrix. 
As a consequence, we give partial answers to a conjecture by Eldredge, Gordina and Saloff-Coste; namely, we give large subsets $\mathcal S$ of the space of left-invariant metrics $\mathcal M$ on $G$ such that there exists a positive real number $C$ depending on $G$ and $\mathcal S$ such that $\lambda_1(G,g)\operatorname{diam}(G,g)^2\leq C$ for all $g\in\mathcal S$. 
The existence of the constant $C$ for $\mathcal S=\mathcal M$ is the original conjecture.
\end{abstract}

\maketitle

%\tableofcontents
	
\section{Introduction}\label{sec:intro}

The diameter of a compact Riemannian manifold is, curiously, a geometric object easily defined but of extreme difficulty to compute explicitly.
Similarly, the smallest positive eigenvalue of the Laplace--Beltrami operator is a very important and highly studied object, which is generically not computable and is known only in very special cases.

These objects have been good friends for a long time, sharing many articles and formulas. 
For instance, many of the most important estimates for the first Laplace eigenvalue are in terms of the diameter (e.g.\ \cite{Cheng75}, \cite{LiYau80}, \cite{ZhongYang84}, \cite{Yang99}; see \cite[\S9.10]{Berger-panoramic}, \cite[\S2.1]{LingLu10}, \cite[\S{}III.3--4]{SchoenYau-book}, \cite[\S3.2 and \S4.3]{Urakawa-book} for some summaries). 
Most of them are (positive) lower or upper bounds of $\lambda_1(M,g)\,\diam(M,g)^2$, under geometric conditions on $(M,g)$, usually involving a lower bound for the Ricci curvature. 
Note that the term $\lambda_1(M,g)\,\diam(M,g)^2$ is invariant by homotheties.

Our purpose is to provide estimates for the diameter and the first Laplace eigenvalue in a particular class of compact homogeneous Riemannian manifolds, namely, compact connected Lie groups endowed with left-invariant metrics.

\subsection{Estimates}

Let $G$ be a compact connected Lie group with Lie algebra $\fg$ and dimension $m$. 
Let $\mathcal M^G$ denote the space of left-invariant metrics on $G$. 
It is well known that the elements in $\mathcal M^G$ are in correspondence with inner products on $\fg$. 
Let $g_0$ be a bi-invariant metric on $G$.
We denote by $\innerdots_0$ its corresponding inner product on $\fg$, which is $\Ad(G)$-invariant.

For $g\in\mathcal M^G$, let $\innerdots_g$ denote the corresponding inner product on $\fg$. 
There is a positive definite $\innerdots_0$-self-adjoint linear map $\Omega_g:\fg\to\fg$ satisfying
\begin{equation}\label{eq1:Omega}
\inner{X}{Y}_g = \inner{\Omega_g(X)}{Y}_0
\qquad\text{for all }X,Y\in\fg. 
\end{equation}
We denote by $\sigma_1(g)^2,\dots,\sigma_m(g)^2$ the eigenvalues of $\Omega_g^{-1}$. 
We will always assume 
\begin{equation}
\sigma_1(g)\geq \dots\geq \sigma_m(g)>0.
\end{equation}

It is important to note that the functions $g\mapsto \diam(G,g)\, \sigma_k(g)$ and $g\mapsto {\lambda_1(G,g)}{\sigma_k(g)^{-2}}
$ form $\mathcal M^G$ to $\R_{>0}$ are invariant by homotheties, for any index $1\leq k\leq m$.

It is not difficult to observe that 
\begin{align*}
	\frac{\diam(G,g_0)}{\sigma_1(g)}
	&\leq \diam(G,g) \leq 
	\frac{\diam(G,g_0)}{\sigma_m(g)},\\
	\lambda_1(G,g_0)\,\sigma_m(g)^2 
	&\leq \rule{7pt}{0mm} \lambda_1(G,g) \rule{6pt}{0mm} \leq 
	\lambda_1(G,g_0)\, \sigma_1(g)^2,
\end{align*}
for every $g\in\mathcal M^G$ (see \eqref{eq3:diam-simple-estimates} and \eqref{eq4:lambda1-simple-estimates}). 
A natural aim is to find:
\begin{itemize}
\item a positive lower bound for $\diam(G,g)\, \sigma_k(g)$ with $k$ as large as possible;

\item a positive upper bound for $\diam(G,g)\, \sigma_k(g)$ with $k$ as small as possible;

\item a positive lower bound for $\lambda_1(G,g)\,\sigma_k(g)^{-2}$ with $k$ as small as possible;

\item a positive upper bound for $\lambda_1(G,g)\,\sigma_k(g)^{-2}$ with $k$ as large as possible;
\end{itemize}
in all cases, the bound should hold uniformly for all $g\in\mathcal M^G$. 

The first main result of the article, which summarizes Theorems~\ref{thm3:optimal-index-lower-bound}, \ref{thm3:optimal-index-upper-bound}, \ref{thm4:optimal-index-upper-bound}, and \ref{thm4:optimal-index-lower-bound}, 
shows that the existence of these positive bounds are not always possible, and furthermore, it determines the optimal value of the index $k$ (up to two choices in some cases) for the existence of the positive bound, provided a technical condition holds in certain cases.

\begin{theorem}\label{thm:main0}
Let $G$ be a non-abelian compact connected Lie group of dimension $m$ and Lie algebra $\fg$. 
We set $k_{\max}=1+\max_H\dim H$, where the maximum is taken over all closed subgroups $H$ of $G$ of dimension strictly less than $m$. 
Then, there are positive real numbers $C_1,C_2,C_3,C_4$ depending only on $G$ and $g_0$ such that 
\begin{align}
\label{eq1:diam}
	\frac{C_1}{\sigma_2(g)}
	&\leq \diam(G,g) \leq 
	\frac{C_2}{\sigma_{k_{\max}}(g)},\\
\label{eq1:lambda1}
	C_3\,\sigma_{k_{\max}}(g)^2 
	&\leq \rule{7pt}{0mm} \lambda_1(G,g) \rule{6pt}{0mm} \leq 
	C_4\, \sigma_2(g)^2,
\end{align}
for every $g\in\mathcal M^G$. 
(The existence of the constants $C_2$ and $C_3$ when $G$ is semisimple depends\footnote{See the footnote in page 17 for an update.} on the validity of Condition~\ref{claim}.)
Furthermore, 
\begin{align}
\inf_{g\in\mathcal M^G} \diam(G,g)\,\sigma_4(g) 
=
\inf_{g\in\mathcal M^G} {\lambda_1(G,g)}\,{\sigma_{k_{\max}-1}(g)^{-2}}
&=0,
\\
\sup_{g\in\mathcal M^G} \diam(G,g)\,\sigma_{k_{\max}-1}(g)
=
\sup_{g\in\mathcal M^G} {\lambda_1(G,g)}\, {\sigma_4(g)^{-2}} &=\infty.
\end{align}
Consequently, $k_{\max}$ is the smallest index $k$ satisfying $\sup_{A\in\GL(m,\R) }\diam(G,g_A)\, \sigma_{k}(A)<\infty$ and also $\inf_{g\in\mathcal M^G} {\lambda_1(G,g)}{\sigma_{k}(g)^{-2}}>0$ and, the largest index $k$ satisfying $\inf_{g\in\mathcal M^G} {\lambda_1(G,g)}\,{\sigma_{k}(g)^{-2}}>0$ and $\sup_{g\in\mathcal M^G} {\lambda_1(G,g)}\, {\sigma_k(g)^{-2}}<\infty$ is equal to $k=2$ or $3$. 
\end{theorem}

We will also prove at the end of Sections~\ref{sec:diam} and \ref{sec:eigenvalues} similar estimates valid for a restricted subset of $\mathcal M^G$. 
They will be useful to give partial answers to a conjecture by Eldredge, Gordina, and Saloff-Coste. 

The starting point of all these results are Propositions~\ref{prop3:diam-k} and \ref{prop4:lambda1-k}, which give estimates for $\diam(G,g)$ and $\lambda_1(G,g)$ respectively, by using left-invariant sub-Riemannian and singular Riemannian structures on $G$. 

\subsection{EGS conjecture}
The results introduced so far show that the terms $\diam(G,g)^{-2}$ and $\lambda_2(G,g)$ share a quite similar behavior. 
In fact, every estimate for any of them has a counterpart for the other. 
We next observe that this relation is reasonable.

For any compact homogeneous Riemannian manifold $(M,g)$, 
Peter Li~\cite{Li80} proved that 
\begin{equation}\label{eq1:Li-estimate}
\lambda_1(M,g) \geq \frac{\pi^2/4}{\diam(M,g)^2}.
\end{equation} 
Recently, Judge and Lyons~\cite{JudgeLyons17} improved it. 
Recall that a Riemannian manifold is called \emph{homogeneous} if its isometry group acts transitively on it.
Lie groups endowed with left-invariant metrics form an  important class of homogeneous Riemannian manifolds.
In fact, for any $g\in\mathcal M^G$, the action of $G$ on $(G,g)$ given by multiplication at the left is (obviously) transitive and isometric.

A first evidence of the connection between the functions $\diam(G,g)^{-2}$ and $\lambda_2(G,g)$ mentioned above is that \eqref{eq1:Li-estimate} converts upper bounds for $\diam(M,g)\sigma_k(g)$ or $\lambda_1(M,g)\sigma_k(g)^{-2}$ in lower bounds for the other, for any index $k$. 
More precisely, for $1\leq k\leq m$, $C>0$ and $g\in\mathcal M^G$, 
\begin{align*}
\diam(G,g)\, \sigma_k(g) \leq C
\quad &\Longrightarrow \quad 
\lambda_1(G,g)\, \sigma_k(g)^{-2} \geq \frac{\pi^2}{4C^2},
\\
\lambda_1(G,g)\, \sigma_k(g)^{-2} \leq C
\quad &\Longrightarrow \quad 
\diam(G,g)\, \sigma_k(g) \geq \frac{\pi}{2\sqrt{C}}. 
\end{align*}

In contrast to the lower bound in \eqref{eq1:Li-estimate}, it is easy to see that there is no uniform upper bound for the term $\lambda_1(M,g)\,\diam(M,g)^2$ among compact homogeneous Riemannian manifolds $(M,g)$. 
In fact, $\lambda_1(S^d,g_{\textrm{round}}) \diam(S^d,g_{\textrm{round}})^2 = d\pi^2\to\infty$ when $d\to\infty$. 
Eldredge, Gordina and Saloff-Coste have recently conjectured the existence of a uniform upper bound for $\mathcal M^G$. 

\begin{conjecture}\cite[(1.2)]{EldredgeGordinaSaloff18} \label{conj:EGS}
For any compact connected Lie group $G$, there is a positive real number $C$ depending only on $G$ such that
\begin{equation}\label{eq:EGSconj}
\lambda_1(G,g)  \leq \frac{C}{\diam(G,g)^2}
\qquad\text{for every $g\in\mathcal M^G$.}
\end{equation}
\end{conjecture}

As an abuse of the language, given a particular compact connected Lie group $G$, we will say that the \emph{EGS conjecture holds} for $G$ if there is $C=C(G)>0$ satisfying \eqref{eq:EGSconj}. 
Conjecture~\ref{conj:EGS} claims that the EGS conjecture holds for \emph{every} $G$.

Eldredge, Gordina and Saloff-Coste proposed a detailed method to establish this conjecture.  
They proved that the EGS conjecture holds for every \emph{uniformly doubling} compact connected Lie group (see \cite[Thm.~8.5]{EldredgeGordinaSaloff18}), that is, a compact connected Lie group $G$ satisfying 
\begin{equation}\label{eq:doubling-constant}
\sup_{g\in \mathcal M^G} \;\sup_{r>0}\;  \frac{\op{vol}(B_g(x,2r))}{\op{vol}(B_g(x,r))}< \infty,
\end{equation}
where $B_g(x,r)$ denotes the ball in $(G,g)$ centered at $x$ with radius $r$. 
They in fact conjectured that every compact connected Lie group is uniformly doubling (see \cite[Conj.~1.1]{EldredgeGordinaSaloff18}).
Furthermore, they obtained several analytical consequences for uniformly doubling compact connected Lie groups, including a uniform Poincaré inequality, uniform heat kernel estimates, uniform Harnack inequalities, a uniform gradient estimate, among other results (see \cite[\S8]{EldredgeGordinaSaloff18}).

Since flat tori are uniformly doubling, the EGS conjecture holds for them. 
Eldredge, Gordina and Saloff-Coste proved in addition that $\SU(2)$ is uniformly doubling (see \cite[Thm.~1.2]{EldredgeGordinaSaloff18}), obtaining that the EGS conjecture holds for $\SU(2)$. 
As a consequence of explicit expressions for $\lambda_1(\SU(2),g)$ and $\lambda_1(\SO(3),g)$ for any left-invariant metric $g$, it was obtained in \cite[Thm.~1.4]{Lauret-SpecSU(2)} the following estimates: 
\begin{align}
	\frac{\pi^2}{\diam(\SU(2),g)^2} &<\lambda_1(\SU(2),g)  \leq \frac{8\pi^2}{\diam(\SU(2),g)^2}
\quad \qquad\text{for all $g\in\mathcal M^{\SU(2)}$,} 
\\ \rule{0pt}{20pt}
	\frac{\pi^2}{\diam(\SO(3),g)^2} &<\lambda_1(\SO(3),g)  \leq \frac{(9-4\sqrt{2})\pi^2}{\diam(\SO(3),g)^2}
\quad \qquad\text{for all $g\in\mathcal M^{\SO(3)}$.} 
\end{align}

To the best of the author's knowledge, EGS conjecture is known to be valid only for the groups just reviewed, namely, tori, $\SU(2)$, and $\SO(3)$. 
Because of this, it seems reasonable to consider weaker versions of Conjecture~\ref{conj:EGS} by restricting the class of metrics where the estimate holds. 

For any compact connected simple Lie group $G$, the author showed in \cite{Lauret-natred} that there is $C=C(G)>0$ satisfying that $\lambda_1(G,g)\diam(G,g)^2  \leq C$ for all naturally reductive left-invariant metric $g$ on $G$. 
Naturally reductive metrics form a small and geometrically distinguished subclass of metrics in $\mathcal M^G$, thus this result is not really a strong evidence of Conjecture~\ref{conj:EGS}. 

The next result establishes a weaker version of Conjecture~\ref{conj:EGS} valid for a large subset of $\mathcal M^G$.

\begin{theorem}\label{thm:main}
Let $G$ be an $m$-dimensional compact connected semisimple Lie group with Lie algebra $\fg$, and let $g_0$ be an $\Ad(G)$-invariant inner product on $\fg$.
Let $Y$ be a non-zero element in $\fg$ and let $\fa$ be a real subspace of $\fg$ such that $Y\perp_{g_0} \fa=0$, $\fa$ is contained in a proper subalgebra of $\fg$, and $\fa\cup\{Y\}$ is not contained in a proper subalgebra of $\fg$. 
Write $\fb=(\fa\cup\{Y\})^{\perp_{g_0}}$, thus $\fg=\fa\oplus \R Y\oplus \fb$, and let $\mathcal M^G(\fa,Y)$ denote the set of left-invariant metrics $g$ whose corresponding $\innerdots_0$-self-adjoint map $\Omega_g:\fg\to\fg$ as in \eqref{eq1:Omega} satisfies $\Omega_g (Y) =\sigma Y$, $\Omega_g(\fa)=\fa$, $\Omega_g(\fb)=\fb$, and the smallest (resp.\ largest) eigenvalue of $\Omega_g|_{\fa}$ (resp.\ $\Omega_g|_{\fb}$) is $\geq\sigma$ (resp.\ $\leq \sigma$).

Then, there is $C=C(G,g_0,\fa,Y)>0$ such that 
\begin{equation}
\lambda_1(G,g)\leq \frac{C}{\diam(G,g)^2}
\qquad \text{for all $g\in \mathcal M^G(\fa,Y)$. }
\end{equation}
\end{theorem}

A refined (and more clear) statement of this result is in Theorem~\ref{thm5:EGSweak}. 
We will see in Remark~\ref{rem5:dimension} that the order of growth when $m\to\infty$ of $\dim \mathcal M^G(\fa,Y)$ is the same as for $\dim \mathcal M^G$, namely, $O(m^2)$.
We next give a weaker but cleaner statement.

\begin{corollary}\label{cor:main}
Let $G$ be a compact connected semisimple Lie group with Lie algebra $\fg$, let $g_0$ be an $\Ad(G)$-invariant inner product on $\fg$, and let $\mathcal B:= \{Y_1,\dots,Y_m\}$ be any orthonormal basis of $\fg$ with respect to $g_0$. 
Then, there is $C=C(G,g_0,\mathcal B)>0$ such that 
\begin{equation}
\lambda_1(G,g)\leq \frac{C}{\diam(G,g)^2}
\end{equation}
for every $g\in \mathcal M^G(\mathcal B):=\{g\in\mathcal M^G: g(Y_i,Y_j)=0 \text{ for all } i\neq j \}$. 
\end{corollary}

We will prove this result by showing that $\mathcal M^G(\mathcal B)$ is included in a finite union of sets of the form $\mathcal M^G(\fa,Y)$ as in Theorem~\ref{thm:main}. 
Note that $\dim \mathcal M^G(\mathcal B)=m$. 
Furthermore, for any $g\in\mathcal M^G$, there is an orthonormal basis $\mathcal B$ of $(\fg,g_0)$ such that $g\in\mathcal M^G(\mathcal B)$. 
This follows from the fact that any two positive symmetric matrices commuting to each other can be diagonalized simultaneously.

\subsection{Previous results}
We next review related estimates for the diameter and the first Laplace eigenvalue on compact homogeneous Riemannian manifolds.

Let $G$ be a compact Lie group and let $K$ be a closed subgroup of $G$. 
Let $\fg$ and $\fk$ denote their Lie algebras.
Let $g_0$ be a bi-invariant metric on $G$. 
The manifold $G/K$ endowed with a $G$-invariant metric is a compact homogeneous Riemannian manifold. 
The $G$-invariant metrics on $G/K$ are in correspondence with $\Ad(K)$-invariant inner products on the complement $\fp$ of $\fk$ with respect to $g_0$. 
Consequently, the terms $\sigma_1(g),\dots,\sigma_m(g)$ can be analogously defined in this context.

The diameter of a compact homogeneous Riemannian manifold has been considered in several articles (e.g.\ \cite{Sugahara, FreedmanKitaevLurie, Yang07, Yang08, Kliemann19}). 
A special attention deserves the recent thesis \cite{Kliemann19} by Kliemann. 
He studied local minima of the functional $\mathcal M^G\ni g\mapsto \diam(G,g)^m/\vol(G,g)$. 
We now focus on estimates for $\diam(G/K,g)$ in terms of $\sigma_k(g)$. 

In \cite[Lem.~7.1]{EldredgeGordinaSaloff18}, it was shown that the function $\mathcal M^{\SU(2)}\ni g\mapsto \diam(\SU(2),g)\, \sigma_2(g)$ is bounded on both sides by positive numbers. 
On the other hand, the articles \cite{PodobryaevSachkov16} and \cite{Podobryaev17} obtain explicit expressions for $\diam(\SU(2),g)$ and $\diam(\SO(3),g)$ provided that at least two elements in $\{\sigma_1(g),\sigma_2(g),\sigma_3(g)\}$ coincide. 
Each of these metrics is homothetic to a Berger $3$-sphere.  
As a consequence, one obtains explicit uniform bounds for $\diam(\SU(2),g)\, \sigma_2(g)$ and $\diam(\SO(3),g)\, \sigma_2(g)$
(see \cite[Cor.~4.4]{Lauret-SpecSU(2)}).

In the best of the author's knowledge, there are no more uniform diameter estimates of a compact homogeneous Riemannian manifolds (in terms of the functions $\sigma_1(g),\dots, \sigma_m(g)$) in the literature.

We now move to Laplace eigenvalue estimates of compact homogeneous Riemannian manifolds. 
There is a well-known Lie theoretical procedure to determine the spectrum of a normal homogeneous space (see e.g.\ \cite[\S5.6]{Wallach-book}). 
For instance, \cite[Appendix]{Urakawa86} collects the computations for the first eigenvalue of all compact irreducible symmetric spaces. 

Urakawa was a pioneer on considering the first eigenvalue of non-normal homogeneous spaces (see \cite{Urakawa79, MutoUrakawa80, Urakawa86}).
For instance, he proved (see \cite[Thm.~ 3]{Urakawa79}) that
\begin{equation}
\lambda_1(G,g)\leq \lambda_1(G,g_0)\, \sum_{j=1}^m \sigma_j(A)^2
\qquad\text{for all $g\in\mathcal M^G$}. 
\end{equation}
(We observe in Remark~\ref{rem4:Urakawa} that the simple estimate $\lambda_1(G,g) \leq 
\lambda_1(G,g_0)\, \sigma_1(g)^2$ mentioned above improves this result.)
Furthermore, he obtained explicit expressions for $\lambda_1(G/K,g_t)$ for particular curves of $G$-invariant metrics on $G/K$.

In \cite{Lauret-SpecSU(2)}, the author obtained an explicit expression for $\lambda_1(\SU(2),g)$ and $\lambda_1(\SO(3),g)$ in terms of $\sigma_1(g),\sigma_2(g),\sigma_3(g)$. 
Previously, Urakawa~\cite[Thm.~5]{Urakawa79} had determined such expression for any Berger $3$-sphere (i.e.\ those metrics where at least two of the parameters $\sigma_1(g),\sigma_2(g),\sigma_3(g)$ coincide).

Bringing together the works \cite{BettiolPiccione13a} by Bettiol and Piccione and \cite{BLPhomospheres} by Bettiol, Piccione and the author, one has an explicit expression for the first Laplace eigenvalue of any simply connected symmetric space of real rank one (i.e.\ spheres and complex, quaternionic and the octonionic projective spaces) endowed with an arbitrary homogeneous metric.

\subsection*{Organization}
Section~\ref{sec:preliminaries} recalls the (implicit) description of the spectrum of a compact homogeneous Riemannian manifold. 
It also includes some estimates for the diameter and first Laplace eigenvalue of some left-invariant non-Riemannian structures on a compact Lie group.
Section~\ref{sec:diam} and \ref{sec:eigenvalues} establish the estimates for the diameter and the first Laplace eigenvalue respectively. 
The consequences of these estimates on the EGS conjecture are given in Section~\ref{sec:upperbounds}.
This section ends with some incomplete ideas for solving this conjecture.

\subsection*{Acknowledgments}
The author is grateful for helpful and motivating conversations with 
Renato Bettiol, 
Yves de Cornulier,
Nate Eldredge, 
Lenny Fukshansky, 
Fernando Galaz-Garc\'ia, 
Jorge Lauret, 
Enrico Le Donne,
Juan Pablo Rossetti,
Michael Ruzhansky, 
Dorothee Schueth, and 
Ovidiu Cristinel Stoica.
The author is greatly indebted to the referee for a careful reading and for providing a counterexample of a conjecture in  the first submitted version of the article.

\section{Preliminaries}\label{sec:preliminaries}

In this section we fix a parameterization between left-invariant metrics on a compact Lie group of dimension $m$ and the space of $m\times m$ positive definite real symmetric matrices. 
Then, we recall the well-known description of the spectrum of the Laplace--Beltrami operator associated to an arbitrary left-invariant metric. 
We conclude with a study of the diameter and the first eigenvalue of the Laplacian associated to two left-invariant non-Riemannian structures: sub-Riemannian manifolds and singular Riemannian manifolds.
Although the results in Subsection~\ref{subsec:diam-non-Riemannian} and \ref{subsec:spec-non-Riemannian} are very simple, some of them might not be present in the literature.

Throughout the article, we assume that $G$ is an $m$-dimensional  compact connected Lie group with Lie algebra $\mathfrak g$ and $m\geq2$.

\subsection{Left-invariant metrics}\label{subsec:left-invmetrics}
It is well known that the left-invariant metrics on $G$ are in correspondence with inner products on $\mathfrak g$. 
We next parameterizes this correspondence.
We denote by $I$ the $m\times m$ identity matrix. 

Let $g_I(\cdot,\cdot)$ be an $\Ad(G)$-invariant inner product on $ \mathfrak g$, that is, $g_I(\Ad(a)\cdot X,\Ad(a)\cdot Y)=g_I(X,Y)$ for all $X,Y\in\mathfrak g$ and $a\in G$. 
For instance, a negative multiple of the Killing form provided $\mathfrak g$ is semisimple. 
We fix an orthonormal ordered basis 
\begin{equation}
\mathcal B:=\{X_1,\dots,X_m\}
\end{equation}
of $\mathfrak g$ with respect to $g_I$. 
\begin{quotation}
	{\it Most of the forthcoming definitions in this article will depend on $g_I$ and $\mathcal B$. 
	}
\end{quotation}

\begin{definition}\label{def2:g_A}
We associate to $A=(a_{i,j})_{i,j=1}^m\in\GL(m,\R)$ the following objects:
\begin{itemize}
\item the linear transformation $T_A:\mathfrak g\to \mathfrak g$ determined by $T_A(X_i)=\sum_{j=1}^m a_{i,j} X_j$;

\item the elements $X_j(A)=T_{A^t}(X_j)=\sum_{i=1}^m a_{i,j} X_i$ for $1\leq j\leq m$;

\item the ordered basis $\mathcal B(A):=\{X_1(A),\dots,X_m(A)\}$ of $\mathfrak g$;

\item the inner product $g_A(\cdot,\cdot)$ on $\mathfrak g$ with orthonormal basis $\mathcal B(A)$. 
\end{itemize}
\end{definition}

Clearly $T_I=\Id_{\mathfrak g}$ (the identity map on $\fg$), thus $X_j(I)=X_j$ for all $j$, $\mathcal B(I)=\mathcal B$, and consequently, the notation $g_I$ for the original inner product on $\mathfrak g$ is consistent. 
We will abbreviate $A^{-t}=(A^{-1})^t= (A^t)^{-1}$ for any $A\in\GL(m,\R)$.

It is well known that $\vol(G,g_A)= \vol(G,g_B)$ if and only if $\det(A)=\det(B)$.

\begin{lemma}\label{lem2:g_AP=g_A}
Let $A\in\GL(m,\R)$. 
We have that 
\begin{equation}
g_A(X,Y) = g_I(T_{A^{-t}}(X), T_{A^{-t}}(Y))
\qquad\text{for all $X,Y\in\mathfrak g$.}
\end{equation}
Furthermore, $g_{AP}=g_A$ for all $P\in\Ot(m)$. 
\end{lemma}

\begin{proof}
For $1\leq i,j\leq m$, we have that 
\begin{multline*}
		g_I(T_{A^{-t}}(X_i(A)), T_{A^{-t}}(X_j(A)))
		= \sum_{k,l=1}^m a_{k,i}a_{l,j}\; g_I(T_{A^{-t}}(X_k), T_{A^{-t}}(X_l)) \\
		= \sum_{k,l=1}^m a_{k,i}a_{l,j}\sum_{r,s=1}^m (A^{-t})_{k,r} (A^{-t})_{l,s} \; g_I(X_r, X_s) 
		= \sum_{r=1}^m   \left(\sum_{k=1}^m (A^{-1})_{r,k} a_{k,i}  \right)\left(\sum_{l=1}^m (A^{-1})_{r,l} a_{l,j}  \right)  \\ 
		= \sum_{r=1}^m   (A^{-1}A)_{r,i} (A^{-1}A)_{r,j} 
		= \delta_{i,j} =g_A(X_i(A),X_j(A)),
\end{multline*}
and the first assertion follows. 
We now prove the second assertion by checking that $\mathcal B(AP)$ is an orthonormal basis of $\mathfrak g$ with respect to $g_{A}(\cdot,\cdot)$. 
We have that
\begin{align*}
g_A(X_i(AP), X_j(AP)) 
&= \sum_{k,l=1}^m(AP)_{k,i}(AP)_{l,j} \; g_I(T_{A^{-t}}(X_k),T_{A^{-t}}(X_l)) \\ 
&= \sum_{k,l=1}^m(AP)_{k,i}(AP)_{l,j} \sum_{r,s=1}^m (A^{-t})_{k,r} (A^{-t})_{l,s} \; g_I(X_r,X_s) \\ 
&= \sum_{r=1}^m \left(\sum_{k=1}^m (A^{-1})_{r,k} (AP)_{k,i}  \right)\left(\sum_{l=1}^m (A^{-1})_{r,l} (AP)_{l,j}  \right) 
\\ 
&= \sum_{r=1}^m P_{r,i} P_{r,j} 
= (P^t P)_{i,j}=\delta_{i,j},
\end{align*}
for all $1\leq i,j\leq m$, as asserted.
\end{proof}

\begin{remark}\label{rem2:easy-consequences}
Some easy consequences of Lemma~\ref{lem2:g_AP=g_A} are the following:  
\begin{itemize}
\item[(i)] The $m\times m$-matrix whose $(i,j)$-index is $g_A(X_i,X_j)$ is given by 
\begin{equation*}
[g_A(X_i,X_j)]_{(i,j)} = A^{-t}A^{-1}= (A^{-1})^tA^{-1} = (AA^t)^{-1} .
\end{equation*}
For instance, when $A$ is diagonal, enlarging all the diagonal elements of $A$ shrinks the Riemannian manifold $(G,g_A)$.

\item[(ii)] Since any inner product on $\mathfrak g$ is of the form $g_A$ for some $A\in\GL(m,\R)$, the space of left-invariant metrics on $G$ is identified with $\GL(m,\R) / \Ot(m)$. 

\item[(iii)] $\mathcal B(P)=\{X_1(P),\dots,X_m(P)\}$ is an orthonormal basis of $\mathfrak g$ with respect to $g_I$, for any $P\in\Ot(m)$. 

\item[(iv)] For $P\in\Ot(m)$, $D=\diag(d_1,\dots,d_m)\in \GL(m,\R)$, and any index $1\leq j\leq m$, we have that 
\begin{align*}
	X_j(PD) 
	&= \sum_{i=1}^m (PD)_{i,j} \, X_i 
	= \sum_{i=1}^m p_{i,j}d_j  \, X_i 
	= d_j\, X_j(P).
\end{align*}

\item[(v)] For any $A\in\GL(m,\R)$, there are $P\in\Ot(m)$ an $D=\diag(d_1,\dots,d_m)\in \GL(m,\R)$ such that $AA^t=PD^2P^t$.
Thus $g_{A}=g_{PD}$, and consequently $\{X_1(P),\dots,X_m(P)\}$ is an orthogonal basis for $g_I$ and $g_A$ simultaneously. 
\end{itemize}
\end{remark}

\begin{notation}\label{not2:sigma_j(A)}
For $A\in\GL(m,\R)$, we denote by $\sigma_1(A)^2,\dots, \sigma_1(A)^2$ the eigenvalues of the positive definite symmetric matrix $AA^t$. 
We will always assume that 
\begin{equation*}
\sigma_1(A)\geq \dots\geq \sigma_m(A)>0.
\end{equation*}
We set $D(A)=\diag(\sigma_1(A),\dots,\sigma_m(A))$. 
We say that $P\in \Ot(m)$ \emph{sorts $A$} if 
\begin{equation}\label{eq:diagonalizes}
AA^t= P D(A)^2P^t.
\end{equation}
Such a matrix $P$ always exists, and it is never unique since $PR$ also satisfies \eqref{eq:diagonalizes} for every diagonal matrix $R$ with diagonal coefficients $\pm1$.
Moreover, there exist continuous curves of rotations sorting $A$ when at least one eigenvalue of $AA^t$ is repeated. 
\end{notation}

\begin{remark}
It is clear that the association $g_A \mapsto (\sigma_1(A),\dots,\sigma_m(A)) $ is well defined. 
Moreover, it depends on $g_I$, but not on $\mathcal B$. 
\end{remark}

For $A,B\in \GL(m,\R)$, the matrices $AA^t$ and $BB^t$ are positive definite symmetric matrices. 
We write $AA^t\leq BB^t$ when $BB^t-AA^t$ is a positive semi-definite symmetric matrix, or equivalently, the eigenvalues of $BB^t-AA^t$ are all non-negative. 

\begin{lemma}\label{lem2:A^tAleqB^tB}
	Let $A,B\in \GL(m,\R)$ such that $AA^t\leq BB^t$. 
	Then 
	$
	g_A(X,X)\geq g_B(X,X)
	$
	for all $X\in\mathfrak g$. 
\end{lemma}

\begin{proof}
It is sufficient to show that $ g_B(X_j(A),X_j(A))\leq g_A(X_j(A),X_j(A))=1$ for all $j$. 
	We have that 
	\begin{align*}
		g_B(X_j(A),X_j(A))
		&= \sum_{k,l=1}^m a_{k,j}a_{l,j}\; g_B(X_k,X_l)
		=\sum_{k,l=1}^m a_{k,j}a_{l,j}\; g_I(T_{B^{-t}}(X_k), T_{B^{-t}}(X_l)) \\
		&=\sum_{k,l=1}^m a_{k,j}a_{l,j} \sum_{i,h=1}^m (B^{-t})_{k,i} (B^{-t})_{l,h} \;g_I(X_i,X_h)\\
		&= \sum_{i=1}^m  (B^{-1}A)_{i,j}\; (B^{-1}A)_{i,j} = (A^tB^{-t}B^{-1}A)_{j,j}.
	\end{align*}
It remains to show that $(A^tB^{-t}B^{-1}A)_{j,j} \leq 1 $ for all $j$. 
In fact, $AA^t\leq BB^t$ gives $A^{-t}A^{-1}=(AA^t)^{-1}\geq (BB^t)^{-1}=B^{-t}B^{-1}$, hence $A^tB^{-t}B^{-1}A\leq I$, which implies that the diagonal entries of $A^tB^{-t}B^{-1}A$ are less than or equal to the diagonal entries of $I$, as asserted
\end{proof}

\subsection{Spectra of left-invariant metrics}
\label{subsec:spectraleft-invmetrics}
We denote by $U(\mathfrak g)$ the universal enveloping algebra of $\mathfrak g$.
For $A\in \GL(m,\R)$, we set $C_A=\sum_{j=1}^m X_j(A)^2\in U(\fg)$. 
We have that 
\begin{align}
C_A 
&= \sum_{k=1}^m X_k(A)^2 
=  \sum_{k=1}^m \sum_{i,j=1}^m a_{i,k} a_{j,k} \, X_iX_j 
= \sum_{i,j=1}^m (AA^t)_{i,j}\, X_iX_j .
\notag
\end{align}
Furthermore, one can check that $C_A=\sum_{j=1}^m Y_j^2$ for any other orthonormal basis $\{Y_1,\dots,Y_m\}$ of $\mathfrak g$ with respect to $g_A(\cdot,\cdot)$. 
Remark~\ref{rem2:easy-consequences}(iv) ensures that, if $AA^t=PD^2P^t$ with $P\in\Ot(m)$ and $D=\diag(d_1,\dots,d_m)\in \GL(m,\R)$, then 
\begin{align}
	C_A&= \sum_{j=1}^m d_j^2\, X_j(P)^2.
\end{align}

Let $\pi:G\to \GL(V_\pi)$ be a finite dimensional unitary representation of $G$, and we denote again by $\pi$ to its differential, which is a representation of $\mathfrak g$. 
Let $\langle \cdot,\cdot\rangle_\pi$ denote the inner product on $V_\pi$. 
Since $\pi(a):V_\pi\to V_\pi$ is unitary for every $a\in G$, $\pi(X)$ is skew-hermitian for every $X\in\mathfrak g$, i.e.\ $\langle \pi(X)v,w\rangle_\pi = -\langle v,\pi(X)w\rangle_\pi$ for all $v,w\in V_\pi$. 
Hence $\pi(-X^2)=-\pi(X)\circ\pi(X)$ is self-adjoint and positive semi-definite.
It follows that $\pi(-C_A)$ is self-adjoint and positive semi-definite. 
Moreover, $\pi(-C_A)$ is positive definite when $\pi$ does not have any trivial irreducible component. 
In fact, for any non-trivial irreducible representation $\pi$ of $G$, if $v\in V_\pi$ satisfies $\pi(-C_A)v=0$, then $\pi(X_j(A))v=0$ for all $j$, consequently $\pi(X)v=0$ for all $X\in\mathfrak g$, hence $v=0$ since $\Span_\R\{v\}$ is an invariant subspace of $V_\pi$.

We denote by $\widehat G$ the unitary dual of $G$, that is, the collection of equivalence classes of irreducible unitary representations of $G$. 
For $(\pi,V_\pi)\in\widehat G$, one has the embedding
\begin{equation}\label{eq2:PeterWeyl-embedding}
\begin{aligned} 
	V_\pi\otimes V_\pi^* &\longrightarrow C^\infty(G),\\
	v\otimes\varphi &\longmapsto 
	\big(x\mapsto f_{v\otimes\varphi}(x):= \varphi(\pi(x) v) \big).
\end{aligned}
\end{equation}
Let $\Delta_A$ denote the Laplace--Beltrami operator associated to the Riemannian manifold $(G,g_A)$. 
One has that (c.f.\ \cite[Lem.~1]{Urakawa79}) 
\begin{equation}\label{eq2:Laplacian}
\Delta_A\cdot f_{v\otimes\varphi} = f_{(\pi(-C_A)v)\otimes\varphi}. 
\end{equation}
Suppose that $v\in V_\pi$ is an eigenvector of the finite-dimensional linear operator $\pi(-C_A):V_\pi\to V_\pi$ associated to the eigenvalue $\lambda$, i.e.\ $\pi(-C_A)v=\lambda v$. 	
Then, 
\begin{equation}
\Delta_A \cdot f_{v\otimes\varphi} = f_{(\pi(-C_A)v)\otimes \varphi}
= f_{(\lambda v)\otimes \varphi}=\lambda\, f_{v\otimes\varphi},
\end{equation}
that is, $f_{v\otimes \varphi}$ is an eigenfunction of $\Delta_A$ with eigenvalue $\lambda$ for every $\varphi\in V_\pi^*$. 

We consider on $L^2(G)$ the inner product given by 
\begin{equation}\label{eq2:L^2innerproduct}
\langle f,g\rangle := \int_G f(x)\overline{g(x)}\, dx,
\end{equation}
where $dx$ is a Haar measure on $G$.
It turns out that $L^2(G)$ endowed with this inner product is a Hilbert space. 
The left-regular representation on $L^2(G)$ of $G$, i.e.\ $(a\cdot f)(x)=f(a^{-1}x)$ for $a,x\in G$ and $f\in L^2(G)$, is unitary. 
The Peter-Weyl Theorem ensures that the left-regular representation decomposes as
\begin{equation}\label{eq2:PeterWeyl}
L^2(G)\simeq \bigoplus_{\pi\in \widehat G} V_\pi\otimes V_\pi^*,
\end{equation}
where the embedding of $V_\pi\otimes V_\pi^*$ in $L^2(G)$ is as in \eqref{eq2:PeterWeyl-embedding}.
The action of an element $a\in G$ on $V_\pi\otimes V_\pi^*$ is given by $a\cdot (v\otimes\varphi )= v\otimes (\pi^*(a)\varphi)$ since 
\begin{equation}\label{eq:GactiononL^2(G)}
(a\cdot f_{v\otimes \varphi})(x) = f_{v\otimes \varphi}(a^{-1}x) = \varphi(\pi(a^{-1})\pi(x) v) = (\pi^*(a) \varphi)(\pi(x)v) = f_{v\otimes (\pi^*(a)\varphi)}(x).
\end{equation}
By the orthogonal relations (see for instance \cite[Cor.~4.10]{Knapp-book-beyond}), it follows that 
\begin{equation}\label{eq2:basisL^2(G)}
\bigcup_{\pi\in\widehat G}\{f_{v_i\otimes \varphi_j}: 1\leq i,j\leq d_\pi\}
\end{equation}
is an orthonormal basis of $L^2(G)$, where for each $\pi\in\widehat G$, 
\begin{itemize}
	\item $d_\pi=\dim V_\pi=\dim V_\pi^*$,
	\item $\{v_1,\dots,v_{d_\pi}\}$ is any orthonormal basis of $V_\pi$, and
	\item $\{\varphi_1,\dots,\varphi_{d_\pi}\}$ is any orthonormal basis of $V_\pi^*$.
\end{itemize}

For each $\pi\in \widehat G$ non-trivial, we take an orthonormal eigenbasis  $\{v_1,\dots,v_{d_\pi}\}$ of $\pi(-C_A)$, i.e.\ 
$ \pi(-C_A)v_i=\lambda_i^{\pi,A} \, v_i$ 
for some $\lambda_i^{\pi,A}>0$. 
We thus obtain that the basis of $L^2(G)$ in \eqref{eq2:basisL^2(G)} contains only eigenfunctions of $\Delta_A$.
Hence,  
\begin{equation}\label{eq2:spec_A}
\Spec(G,g_A):=
\Spec (\Delta_A) = \bigcup_{\pi\in\widehat G} \big\{\!\big\{
\underbrace{\lambda_i^{\pi,A},\dots, \lambda_i^{\pi,A}}_{d_\pi\text{-times}}:1\leq i\leq d_\pi 
\big\}\!\big\}.
\end{equation}
(Here, the double curly brackets is to emphasize that the spectrum is a multiset and not a set.)
The multiplicity $d_\pi$ for each $\lambda_i^{\pi,A}$ above comes from the following fact: $f_{v_i\otimes \varphi_j}$ is an eigenfunction of $\Delta_A$ with eigenvalue $\lambda_i^{\pi,A}$ for every $1\leq j\leq d_\pi$. 

For $\Phi:W\to W$ a linear transformation of a finite-dimensional complex vector space $W$, we denote by $\lambda_{\min}(\Phi)$ its smallest eigenvalue. 
The expression \eqref{eq2:spec_A} yields 
\begin{equation}\label{eq2:lambda1(G,g_A)}
\lambda_1(G,g_A) = \min \left\{ \lambda_{\min}(\pi(-C_A)): \pi\in\widehat G,\, \pi\not\simeq 1_G \right\}. 
\end{equation}

\begin{remark}\label{rem2:C_I}
The case $A=I$ is very particular since $C_I$ lies in the center of $\mathcal U(\mathfrak g)$ (e.g.\ when $\mathfrak g$ is semisimple and $g_I$ is minus the Killing form, then $C_I$ is the \emph{Casimir element}). 
Thus, for any $\pi\in\widehat G$, $\pi(-C_I)$ commutes with $\pi(g)$ for every $g\in G$, and then Schur's Lemma yields that $\pi(-C_I)$ acts by an scalar on $V_\pi$.
By denoting this scalar by $\lambda^\pi$, i.e.\ $\pi(-C_I) =  \lambda^\pi\, \Id_{V_\pi}$, we have that 
\begin{equation}\label{eq:spec_I}
\Spec(G,g_I)=\Spec (\Delta_I) = \bigcup_{\pi\in\widehat G} \big\{\!\big\{
\underbrace{\lambda^{\pi},\dots, \lambda^{\pi}}_{d_\pi^2\text{-times}} 
\big\}\!\big\}.
\end{equation}
\end{remark}

\begin{remark}\label{rem2:G/H}
We will occasionally consider some homogeneous Riemannian spaces of the following form. 
Let $H$ be a closed subgroup of $G$ with Lie algebra $\fh$. 
The space of $G$-invariant metrics on $G/H$ is in correspondence with the set of $\Ad(H)$-invariant inner products on $\fh^\perp =\{X\in\fg: g_I(X,\fh)=0\}$.  

In the sequel, we will mostly consider the particular case $(G/H,g_I|_{\fh^\perp})$, which is a normal homogeneous space. 
The spectrum of its associated Laplace-Beltrami operator is obtained in a similar way as for $(G,g_I)$. 
Namely, 
\begin{equation}\label{eq2:spec(G/H,g_I)}
\Spec (G/H,g_I|_{\fh^\perp}) = \bigcup_{\pi\in\widehat G_H} \big\{\!\big\{
\underbrace{\lambda^{\pi},\dots, \lambda^{\pi}}_{(d_\pi\dim V_\pi^H)\text{-times}} 
\big\}\!\big\},
\end{equation}
where $\widehat G_H$ denotes the set of spherical representations of $(G,H)$, that is, those $\pi\in\widehat G$ satisfying that $V_\pi^H=\{v\in V_\pi: \pi(a)v=v\text{ for all }a\in H\}\neq 0$. 
In fact, $f_{v\otimes \varphi}$ defines a function on $G/H$ (i.e.\ $f_{v\otimes \varphi}(xa)= f_{v\otimes \varphi}(x)$ for all $a\in H$) if and only if $v\in V_\pi^H$, and this explains the reduction of the multiplicity $d_\pi^2$ in \eqref{eq:spec_I} to $d_\pi \dim V_\pi^H$ in \eqref{eq2:spec(G/H,g_I)}.

In particular, we have that $0<\lambda_1(G/H,g_I|_{\fh^\perp}) <\infty$ if and only if $0<\dim \fh <m$. 
\end{remark}

\begin{lemma}\label{lem2:C_AB}
For $A,B\in\GL(m,\R)$, we have that 
$$
\displaystyle C_{AB}= \sum_{i=1}^m  \sum_{j=1}^m \,(BB^t)_{i,j}\, X_i(A)X_j(A). 
$$
\end{lemma}

\begin{proof}
We have that 
\begin{align*}
C_{AB} 
&= \sum_{k,l=1}^m (ABB^tA^t)_{k,l} \, X_kX_l 
= \sum_{k,l=1}^m \sum_{i,j=1}^m a_{k,i} (BB^t)_{i,j} a_{l,j}\, X_kX_l \\
&= \sum_{i,j=1}^m  (BB^t)_{i,j} \left(\sum_{k=1}^m a_{k,i} X_k\right) \left(\sum_{l=1}^m a_{l,j}X_l\right) 
= \sum_{i,j=1}^m  (BB^t)_{i,j} X_i(A)X_j(A),
\end{align*}
as asserted.
\end{proof}

\subsection{Diameter of left-invariant non-Riemannian structures} \label{subsec:diam-non-Riemannian}
Throughout this subsection, $M$ denotes a smooth manifold. 
A Riemannian metric $g$ on $M$ has canonically associated a length for any smooth path on $M$, the distance function $\dist_{(M,g)}(\cdot,\cdot)$ defined by the infimum of the lengths over all smooth paths joining the points, the corresponding metric space $(M,\dist_{(M,g)})$, and the diameter $\diam(M,g) \in [0,\infty]$ given by the supremum of the distances between two points in $M$. 
Clearly, $\diam(M,g)<\infty $ if $M$ is compact and connected. 
More precision on these notions can be found in most textbooks on Riemannian geometry. 

\begin{lemma}\label{lem2:diam-monotonicity}
For Riemannian metrics $g$ and $h$ on $M$ satisfying that $g_p(X,X)\leq h_p(X,X)$ for all $X\in T_pM$ and $p\in M$, we have that $\diam(M,g)\leq \diam(M,h)$.
\end{lemma}

\begin{proof}
We assume that $M$ is connected, otherwise the diameter is $\infty$ for every Riemannian metric on $M$. 
Furthermore, we assume that $M$ is compact, leaving the proof of the general case to the reader. 
 
Since $M$ is compact, there are $p,q\in M$ satisfying that $\diam(M,h) = \dist_{(M,h)}(p,q)$. 
It is well known that there is $\gamma:[0,1]\to M$ a smooth path realizing the distance between $p$ and $q$ with respect to $g$. 
Hence, 
\begin{multline*}
\diam(M,g) 
	\geq \dist_{(M,g)}(p,q) 
	=\op{length}_{(M,g)}(\gamma)
	= \int_0^1 g(\dot\gamma(t),\dot\gamma(t))^{1/2} dt \\
	\geq \int_0^1 h(\dot\gamma(t),\dot\gamma(t))^{1/2} dt 
	\geq \dist_{(M,h)}(p,q)=\diam(M,h),
\end{multline*}
and the proof is complete. 
\end{proof}

\bigskip

A \emph{sub-Riemannian manifold} is a triple $(M,\mathcal D,g)$, where $\mathcal D$ is a subbundle of $TM$ and $g=(g_p)_{p\in M}$ denotes a family of inner product on $\mathcal D$ which smoothly vary with the base point
(see \cite{Montgomery-tour} for a general reference).
A smooth curve $\gamma$ on $(M,\mathcal D,g)$ is called \emph{horizontal} if $\gamma'(t)\in \mathcal D_{\gamma(t)}$ for all $t$. 
The length of a horizontal curve $\gamma:[a,b]\to M$ is equal to $\op{length}_{(M,\mathcal D,g)}(\gamma):=\int_a^b g_{\gamma(t)}({\gamma'(t)},{\gamma'(t)})^{1/2} \,dt$. 
The sub-distance between two points $p,q\in M$ is defined as the infimum of $\op{length}_{(M,\mathcal D,g)}(\gamma)$ over all horizontal curves $\gamma$ on $M$ connecting $p$ and $q$. 
The corresponding diameter, $\diam(M,\mathcal D,g)$, is given by the supremum of the distances between two points in $M$. 
Consequently, the diameter is $\infty$ if two points in $M$ cannot be joined by a horizontal smooth curve. 
The next lemma follows similarly as Lemma~\ref{lem2:diam-monotonicity}.

\begin{lemma}\label{lem2:sub-diam-monotonicity}
Let $\mathcal D$ be a subbundle on $M$. 
For sub-Riemannian structures $g$ and $h$ on $(M,\mathcal D)$ satisfying that $g_p(X,X)\leq h_p(X,X)$ for all $X\in \mathcal D_p$ and $p\in M$, we have that $\diam(M,\mathcal D, g)\leq \diam(M,\mathcal D, h)$.
\end{lemma}

\begin{lemma}\label{lem2:sub-diam-restriccionRiemanniano}
Let $\mathcal D$ be a subbundle on $M$. 
If $g$ is a Riemannian metric on $M$, then the sub-Riemannian metric $h$ on $(M,\mathcal D)$ given by the restriction of $g$ on $\mathcal D$ (i.e.\ $h_p=g_p|_{\mathcal D_p}$ for all $p\in M$) satisfies 
\begin{equation*}
\diam(M,g)\leq \diam(M,\mathcal D,h). 
\end{equation*}
\end{lemma}

Lemma~\ref{lem2:sub-diam-restriccionRiemanniano} follows immediately by noting that the Riemannian geodesic connecting two given points in $M$ may not be horizontal.

We say that a subbundle $\mathcal D$ satisfies the \emph{bracket-generating condition} (also known as the \emph{Hörmander condition}) if the Lie algebra generated by vector fields in $\mathcal D$ spans at every point the tangent space of $M$. 
For such a $\mathcal D$, provided $M$ is compact, the Chow--Rashevskii Theorem ensures that $\diam (M,\mathcal D,h)<\infty$.
In particular, any two points in $M$ can be joined by a horizontal curve.

In what follows we will consider a very particular kind of sub-Riemannian manifolds, namely, a compact Lie group $G$ endowed with a \emph{left-invariant} sub-Riemannian structure.
Given $\HH$ a subspace of $\fg$ and $b(\cdot,\cdot)$ an inner product on $\HH$, we associate the left-invariant sub-Riemannian structure $(\mathcal D,g)$ given by 
\begin{align}\label{eq2:left-inv-sub}
\mathcal D&=\bigcup_{a\in G}dL_a(\HH), &
	g_a\big(dL_a(X),dL_a(Y)\big) &= b({X},{Y}) ,
\end{align}
for all $X,Y\in\HH$ and $a\in G$. 
Here, $L_a:G\to G$ is given by $L_a(x)=ax$ and $\HH$ is seen as a subspace of $T_eG\equiv \fg$. 
We will denote this sub-Riemannian manifold by $(G,\HH,g)$ and, as in the Riemannian case, $g$ will be identified with the inner product $g_e=b$ on $\HH$.

\begin{definition}\label{def2:bracket-generating}
A subset $S$ of $\mathfrak g$ is called \emph{bracket generating} if the Lie algebra generated by $S$ is equal to $\mathfrak g$.
Equivalently, the only subalgebra of $\mathfrak g$ containing $S$ is $\mathfrak g$.
\end{definition}

Of course, a bracket-generating subspace $\HH$ of $\fg$ induces a left-invariant subbundle of $TG$ satisfying the bracket-generating condition. 
The next theorem follows immediately from the Chow--Rashevskii Theorem. 
Since we will encounter the situation of the theorem many times in the course of this paper, we state it here.

\begin{theorem}\label{thm2:Chow}
If $\HH$ is a bracket-generating subspace of $\fg$, then $\diam(G,\HH, g)<\infty$ for any inner product $g$ on $\HH$. 
\end{theorem}

For a general treatment of sub-Riemannian geometry we refer the reader to \cite{AgrachevBarilariBoscain-book}, \cite{LeDonne-lecturenotes}, and \cite{Montgomery-tour}. 
A brief account on left-invariant sub-Riemannian structures on compact Lie groups can be found in \cite[\S9]{EldredgeGordinaSaloff18} (see also \cite[Ch.~7]{AgrachevBarilariBoscain-book}). 
In the present article we will only use the few facts just reviewed.

\bigskip

We now introduce the second non-Riemannian structure. 
Given $g=(g_p)_{p\in M}$ such that $g_p$ is a positive semi-definite symmetric bilinear form on $T_pM$ at each point $p\in M$ varying smoothly, $(M,g)$ is called a \emph{singular Riemannian manifold}. 
See \cite{Kupeli} for the general theory on a more general context: \emph{singular pseudo-Riemannian manifolds} (i.e.\ $g_p$ is any symmetric bilinear form on $T_pM$). 
A word of caution: the name `singular Riemannian manifold' has been used sometimes for different objects, for instance, an `almost-Riemannian manifold'.

The corresponding length of a smooth curve $\gamma:[a,b]\to M$ equals $\int_a^b g_{\gamma(t)}({\gamma'(t)},{\gamma'(t)})^{1/2} \,dt$. 
The \emph{singular distance} between two points $p,q\in M$ is defined as the infimum of the lengths over all smooth curves $\gamma$ on $M$ connecting $p$ and $q$. 
The corresponding diameter, $\diam(M,g)$, is given by the supremum of the distances between two points in $M$. 

\begin{remark}\label{rem2:pseudo-distance}
The corresponding singular distance $\dist_{(M,g)}$ of $(M,g)$ is a pseudo-distance in the sense of \cite[Def.~1.1.4]{BuragoBuragoIvanov-book}, that is, it satisfies all the properties of a distance except the requirement that $\dist_{(M,g)}(p,q)=0$ implies $p=q$. 
Moreover, the singular diameter of a non-trivial singular Riemannian manifold might be zero, such as is shown in 
Example~\ref{ex2:diam-singular=0} below. 
By identifying points in $M$ with zero distance in the pseudo-metric space $(M,\dist_{(M,g)})$, we obtain a metric space that we denote by $(M/\dist_{(M,g)} ,\hat \dist_{(M,g)})$ (see for instance \cite[Prop.~1.1.5]{BuragoBuragoIvanov-book}).
\end{remark}

\begin{notation}\label{not2:symmetricbilinearform}
Given $b$ a (real) symmetric bilinear form on $\fg$ and $\fa$ a (real) subspace of $\fg$, let us denote by $b|_{\fa}$ the symmetric bilinear form on $\fa$ given by the restriction of $b$ on $\fa$, that is, $b|_{\fa}(X,Y)= b(X,Y)$ for all $X,Y\in\fa$. 
Furthermore, when $b$ is non-degenerate, let $b|_{\fa}^*$ denote the symmetric bilinear form on $\fg$ given by $b|_{\fa}^*(X_1+X_2,Y_1+Y_2)=b(X_1,Y_1)$ for all $X_1,Y_1\in\fa$ and $X_2,Y_2\in\fa^{\perp_b}:=\{X\in\fa:b(X,Y)=0\text{ for all }Y\in\fa\}$. 
Note that if $b$ is positive definite, then $b|_{\fa}$ is positive definite and $b|_{\fa}^*$ is positive semi-definite. 
\end{notation}

The next results are analogous to Lemmas~\ref{lem2:sub-diam-monotonicity} and \ref{lem2:sub-diam-restriccionRiemanniano} respectively. 

\begin{lemma}\label{lem2:sing-diam-monotonicity}
For singular Riemannian metrics $g$ and $h$ on $M$ satisfying that $g_p(X,X)\leq h_p(X,X)$ for all $X\in T_pM$ and $p\in M$, we have that $\diam(M,g)\leq \diam(M, h)$.
\end{lemma}

\begin{lemma}\label{lem2:sing-diam-restriccionRiemanniano}
Let $\mathcal D$ be a subbundle on $M$. 
If $g$ is a Riemannian metric on $M$, the singular  Riemannian metric $h$ given by $h_p=g_p|_{\mathcal D_p}^*$ for all $p\in M$ satisfies 
\begin{equation*}
	\diam(M,g)\geq \diam(M,h). 
\end{equation*}
\end{lemma}

We next focus on \emph{left-invariant} singular Riemannian structures on a compact Lie group $G$. 
Let $b$ be a positive semi-definite symmetric bilinear form on $\fg$. 
We associate to $b$ the singular Riemannian metric $g$ on $G$ given by
\begin{align}
g_a\big(dL_a(X),dL_a(Y)\big) &= b(X,Y) ,
\qquad\text{for $X,Y\in T_eG\equiv \fg$ and $a\in G$.}
\end{align}
Similarly as above, we will identify $g$ with the symmetric bilinear form $g_e=b$ on $\fg$.

\begin{remark}\label{rem2:radical}
Given $b$ a non-trivial symmetric bilinear form on $\fg$, any complement $\fa$ in $\fg$ of the radical of $b$, 
$$
\op{rad}(b):=\{X\in \fg: b(X,Y)=0\text{ for all }Y\in\fg\}, 
$$
satisfies that $b|_{\fa}$ is non-degenerate. 
\end{remark}

The next lemma, besides being very useful in the sequel, exemplifies the situation discussed in Remark~\ref{rem2:pseudo-distance}.

\begin{lemma}\label{lem2:diam(G/d)=diam(G/H)}
Let $G$ be a compact Lie group, let $H$ be a closed subgroup of $G$ with Lie algebras $\fg$ and $\fh$ respectively, and let $\fp$ denote the orthogonal complement of $\fh$ in $\fg$ with respect to any $\Ad(G)$-invariant inner product on $\fg$. 
Let $h$ be an $\Ad(H)$-invariant positive semi-definite symmetric bilinear form on $\fg$ with $\op{rad}(h)=\fh$. 
Then, the metric space $(G/\dist_{(G,h)} ,\hat\dist_{(G,h)})$ is isometric (as metric spaces) to the metric space corresponding to the homogeneous Riemannian manifold $(G/H,h|_{\fp})$ (see Remark~\ref{rem2:G/H}).
In particular, 
\begin{equation*}
\diam(G,h) = \diam(G/H,h|_{\fp}). 
\end{equation*}
\end{lemma}

The proof is left to the reader. 
The isometry is given by the map $\hat a\mapsto aH$, where $\hat a$ denotes the class of $a\in G$ in $G/\dist_{(M,h)}$. 
The last identity follows since $\diam(G,h) =\diam(G,\dist_{(M,h)}) = \diam(G/\dist_{(M,h)},\hat\dist_{(M,h)}) = \diam(G/H,h|_{\fp})$.
It is important to note that if $\gamma:[0,1]\to G$ is a smooth path realizing the distance in $(G,h)$ between $\gamma(0)$ and $\gamma(1)$, the smooth path $\widetilde \gamma:[0,1]\to G/H$ given by $\widetilde \gamma(t)=\gamma(t)H$ is not necessarily a geodesic in $(G/H, h|_{\fp})$ since $\widetilde \gamma'$ may vanish in some open interval of $[0,1]$.

\begin{example}\label{ex2:diam-singular=0}
Let $h$ be a (non-trivial) positive semi-definite symmetric bilinear form on $\fg$ such that $\op{rad}(h)$ is bracket generating. 
Given any two points $a,b\in G$, the Chow--Rashevskii Theorem ensures that there is a smooth curve $\gamma$ connected them with $\gamma'(t)\in \op{rad}(h)$ for all $t$.
It follows that the singular distance between $a$ and $b$ is zero since $h(\gamma'(t),\gamma'(t))=0$ for all $t$. 
Hence $\diam(G,h)=0$. 
\end{example}

We conclude from Example~\ref{ex2:diam-singular=0} that a necessary condition to ensure $\diam(G,h)>0$ is that $\op{rad}(h)$ cannot be bracket generating. 
This condition is not sufficient.
For instance, if $G$ is an $m$-dimensional flat torus $T^m$,  $X_1\in\fg $ is chosen so that $H:=\{\exp(tX_1):t\in\R\}$ is dense in $G$, and $\op{rad}(h)=\Span_\R\{X_1\}$ (i.e.\ $h$ is non-degenerate in some complement of $\R X_1$ in $\fg$ and $h(X_1,\fg)=0$), then $\diam(G,h)=0$.
This follows form the fact that any two points in the dense subset $H$ have distance zero. 
However, the next result tells us that a slightly stronger condition works.

\begin{proposition}\label{prop2:diam-pseudo-Hclosed}
Let $(G,h)$ be a left-invariant singular Riemannian manifold induced by a positive semi-definite symmetric bilinear form $h$ on $\fg$. 
If $\op{rad}(h)$ is contained in a proper Lie subalgebra $\mathfrak h$ of $\mathfrak g$ whose associated connected subgroup $H$ of $G$ is closed, then $\diam(G,h)>0$. 
\end{proposition}

\begin{proof}
By assumption, there is a proper closed subgroup $H$ of $G$ such that its Lie algebra $\fh$ contains $\op{rad}(h)$. 
Let $\fp$ be the orthogonal complement subspace of $\fh$ in $\fg$ with respect to any $\Ad(G)$-invariant inner product $g_0$ on $\fg$. 
There is $t>0$ sufficiently small such that $h(X,X)\geq t\, g_0(X,X)$ for all $X\in \fp$. 
Lemma~\ref{lem2:sing-diam-monotonicity} implies that $\diam(G,h)\geq \diam(G,t\,g_0|_{\fp}^*)$. 
Now, Lemma~\ref{lem2:diam(G/d)=diam(G/H)} yields $\diam(G,t\,g_0|_{\fp}^*)= \diam(G/H,(t\,g_0|_{\fp}^*)|_{\fp} ) = \diam(G/H,t\,g_0|_{\fp})$, which is clearly positive, and the proof is complete. 
\end{proof}

\begin{remark}
The assumption in the previous lemma of the existence of a closed subgroup $H$ with a Lie algebra $\fh$ containing $\op{rad}(h)$ and $\fh\neq\fg$, avoids the case that $\op{rad}(h)$ generates a proper Lie subalgebra of $\mathfrak g$ whose connected subgroup of $G$ is dense in $G$. 
Clearly, when $\mathfrak g$ is non-abelian, this assumption always holds if $\dim \op{rad}(h)=1$ since $\{\exp(t X):t\in\R\}$ is contained always in some maximal torus of $G$ for any $X\in \fg$. 
Moreover, when $G$ is semisimple (i.e.\ $[\fg,\fg]= \fg$), the condition is equivalent to $\op{rad}(h)$ is not bracket generating in $\fg$.
This follows from the fact that a semisimple compact Lie group does not have dense proper subgroups (see for instance \cite[Thm.~3.3]{MaciasVirgos}). 
This consequence is stated in the next corollary.
\end{remark}

\begin{corollary} \label{cor2:diam-pseudo-Hclosed-Gss}
Let $G$ be a compact connected semisimple Lie group and let $h$ be a positive semi-definite symmetric bilinear form on $\fg$. 
If $\op{rad}(h)$ is not bracket generating, then $\diam(G,h)>0$. 
\end{corollary}

\subsection{Spectra of left-invariant non-Riemannian structures}
\label{subsec:spec-non-Riemannian}
Let $\HH$ be a subspace of $\fg$ and $h$ an inner product on it. 
The \emph{sub-Laplace operator} (or \emph{sub-Laplacian}) associated to the sub-Riemannian manifold $(G,\HH,h)$ (introduced in \eqref{eq2:left-inv-sub}) is the (positive semi-definite self-adjoint) differential operator on $C^\infty(G)$ given by 
\begin{equation}\label{eq2:subLaplacian}
\Delta_{(\HH,h)}(f)= -\sum_{j=1}^l Y_j^2\cdot f,
\end{equation}
where $\{Y_1,\dots,Y_l\}$ is any orthonormal basis of $\HH$ with respect to the inner product $h$ and $(X\cdot f)(a)=\left.\tfrac{d}{dt}\right|_{t=0} f(\exp(X)a)$ for all $X\in\fg$, and $a\in G$. 
We set $C_{(\HH,h)}=\sum_{j=1}^l Y_j^2 \in\mathcal U(\fg)$.
For $\pi\in\widehat G$, $v\in V_\pi$, $\varphi\in V_\pi^*$, and $f_{v\otimes\varphi}\in C^\infty(G)$ given as in \eqref{eq2:PeterWeyl-embedding}, one has that 
\begin{equation}\label{eq2:L-f_vxvarphi}
\Delta_{(\HH,h)} \cdot f_{v\otimes \varphi} = f_{(-\pi(C_{(\HH,h)}) v)\otimes \varphi}. 
\end{equation}
By proceeding in the same way as for \eqref{eq2:lambda1(G,g_A)}, one gets that the second (possible zero) eigenvalue of $\Delta_{(\HH,h)}$ is given by 
\begin{equation}\label{eq2:lambda_1(G,H,g_I|H)}
\lambda_1(G,\HH,h) = \min \left\{ \lambda_{\min}(\pi(-C_{(\HH,h)})): \pi\in\widehat G,\, \pi\not\simeq 1_G \right\}. 
\end{equation}
By Hörmander's theorem (\cite{Hormander67}), $\Delta_{(\HH,h)}$ is hypoelliptic when $\HH$ is bracket generating. 
In particular, $\Delta_{(\HH,h)}$ has a discrete spectrum since the inverse operator to $1+\Delta_{(\HH,h)}$ is compact.

Although the next result may be obvious, we include a short and self-contained proof. 

\begin{lemma} \label{lem2:sub-spec-discrete}
If $\HH$ is bracket generating, then the eigenvalue $0$ in the spectrum of $\Delta_{(\HH,h)}$ has multiplicity one, i.e.\ $\lambda_1(G,\HH,h)>0$. 
\end{lemma}

\begin{proof}
	From \eqref{eq2:L-f_vxvarphi}, it follows that the multiplicity of $\lambda\geq0$ in the spectrum of $\Delta_{(\HH,h)}$ is 
	$$
	\sum_{\pi\in\widehat G} \dim V_\pi \; \dim \{v\in V_\pi: \pi(-C_{(\HH,h)})v=\lambda v\}.
	$$
	Clearly, the trivial representation $1_G$ of $G$ contributes to the spectrum of $\Delta_{(\HH,h)}$ with the eigenvalue $0$ exactly once. 
	Thus, the assertion is equivalent to show that  $\lambda_{\min}(\pi(-C_{(\HH,h)}))>0$ for every $\pi\in\widehat G \smallsetminus\{1_G\}$.
	
	We fix $\pi_0 \in \widehat G\smallsetminus\{1_G\}$ and suppose that $v_0\in V_{\pi_0}$ satisfies $\pi_0(-C_{(\HH,h)}) v_0=0$. 
	Since $-\pi_0(Y_j)^2\geq0$, we obtain that 
	$
	\pi_0(Y_j) \, v_0 =0
	$ 
	for all $1\leq j\leq k$.
	It follows that 
	$$
	\pi_0([Y_i,Y_j]) \, v_0
	=\big( \pi_0(Y_i) \pi_0(Y_j)-\pi_0(Y_i) \pi_0(Y_j) \big)\,v_0=0
	$$
	for all $1\leq i<j\leq k$. 
	Proceeding in this way, we obtain that $\pi_0(Y)v_0=0$ for all $Y$ in the Lie algebra generated by $\{Y_1,\dots,Y_k\}$, which is $\mathfrak g$ since $\HH$ is bracket generating by assumption.  
	This yields that $v_0=0$ and completes the proof. 
\end{proof}

\begin{remark}\label{rem2:lambda1(G,H,h)=0}
	If $\HH$ is contained in the Lie algebra $\fh$ of a closed connected subgroup $H$ of $G$ (in particular $\HH$ is not bracket generating), then $\lambda_1(G,\HH,h)=0$. 
	In fact, $\Delta_{(\HH,h)}\cdot f=0$ for all $H$-invariant $f\in C^\infty(G)$.
	The subspace of these functions is far from being empty because $L^2(G/H)=\bigoplus_{\pi\in \widehat G_H}( \dim V_\pi^H)\, V_\pi$.
\end{remark}

\section{Diameter estimates} \label{sec:diam}
We assume throughout the section that $G$ is a compact connected Lie group with Lie algebra $\mathfrak g$ of dimension $m$. 
Furthermore, we fix an $\Ad(G)$-invariant inner product $g_I$ on $\fg$ and an orthonormal basis $\mathcal B=\{X_1,\dots,X_m\}$. 
In Subsection~\ref{subsec:left-invmetrics}, we associated to $A\in\GL(m,\R)$ a left-invariant metric $g_A$ on $G$.
We deal in this section with estimates for the diameter of $(G,g_A)$ in terms of the eigenvalues of $AA^t$. 
The information in Subsection~\ref{subsec:diam-non-Riemannian} is very important in this section.

\subsection{Simple estimates for the diameter}

To motivate the diameter estimates of this section, we begin by discussing the simple estimates
\begin{equation}\label{eq3:diam-simple-estimates}
\frac{\diam(G,g_I)}{\sigma_1(A)} 
\leq \diam(G,g_A) \leq 
\frac{\diam(G,g_I)}{\sigma_m(A)}
\qquad\text{for any $A\in\GL(m,\R)$. }
\end{equation}
We recall from Notation~\ref{not2:sigma_j(A)} that $\sigma_1(A)$ and $\sigma_m(A)$ denote the largest and smallest eigenvalue of $AA^t$ respectively.  
This estimate will follow from the next result. 

\begin{lemma}\label{lem3:diam-inequality}
	Let $A,B\in \GL(m,\R)$ satisfying that $AA^t\leq BB^t$. 
	Then
	$$
	\diam(G,g_A)\geq \diam(G,g_B).
	$$
\end{lemma}

\begin{proof}
By Lemma~\ref{lem2:A^tAleqB^tB}, $AA^t\leq BB^t$ forces to  
$
g_A(X,X)\geq g_B(X,X)
$ 
for all $X\in\fg$. 
The proof follows by Lemma~\ref{lem2:diam-monotonicity}. 
\end{proof}

We now prove \eqref{eq3:diam-simple-estimates}.
Let $P$ be any matrix in $\Ot(m)$ sorting $A$ (see Notation~\ref{not2:sigma_j(A)}). 
Then
\begin{equation}\label{eq2:A^tAprimero}
AA^t 
= P \begin{pmatrix} \sigma_1(A)^2\\ &\ddots \\ && \sigma_m(A)^2 \end{pmatrix}P^t
\geq 
P \begin{pmatrix} \sigma_m(A)^2\\ &\ddots \\ && \sigma_m(A)^2 \end{pmatrix}P^t = 
\sigma_m(A)^2\,  I. 
\end{equation}
Lemma~\ref{lem3:diam-inequality} now yields 
$
\diam(G,g_{A}) \leq \diam (G,g_{\sigma_m(A)I}), 
$
and consequently the right-hand side of \eqref{eq3:diam-simple-estimates} follows since $\diam(G,g_{tB}) = \diam(G,t^{-2}g_{B}) = t^{-1}\diam(G,g_B)$ for all $t>0$ and $B\in\GL(m,\R)$. 
The other estimate follows analogously by using $AA^t\leq \sigma_1(A)^2 \, I$.

\subsection{Main tool for the diameter}

Proposition~\ref{prop3:diam-k} below will be the main tool in the rest of the section and it is based on ideas from the proof of \cite[Lem.~7.1]{EldredgeGordinaSaloff18}. 
We require some notation to state it.

\begin{notation}\label{not2:HH_P^kCC_P^k}
	For $P\in\Ot(m)$, one has that $\{X_1(P),\dots,X_m(P)\}$ is an orthonormal basis of $\fg$ with respect to $g_I$ (see Remark~\ref{rem2:easy-consequences}(iii)). 
	For any index $1\leq k\leq m$, we set 
	\begin{equation}
	\HH_{P,k} = \Span_\R\{X_1(P),\dots, X_k(P)\}
	\quad\text{and}\quad
	\CC_{P,k} = \Span_\R\{X_{k}(P),\dots, X_m(P)\}. 
	\end{equation}
\end{notation}

We recall from Subsection~\ref{subsec:diam-non-Riemannian} that the inner product $g_I|_{\HH_{P,k}}$ on $\HH_{P,k}$ has orthonormal basis $\{X_1(P),\dots, X_k(P)\}$ and induces the sub-Riemannian manifold $(G,\HH_{P,k},g_I|_{\HH_{P,k}})$. 
Analogously, $g_I|_{\CC_{P,k}}^* $ denotes the positive semi-definite symmetric bilinear form on $\fg$ determined by $g_I|_{\CC_{P,k}}^*(X_i(P),X_j(P))=1$ for $i=j\geq k$ and zero otherwise, which induces the singular Riemannian manifold $(G,g_I|_{\CC_{P,k}}^*)$.

\begin{proposition}\label{prop3:diam-k}
	Let $A\in \GL(m,\R)$.
	For any $P$ sorting $A$ and any index $1\leq k\leq m$, we have that
\begin{equation}\label{eq3:diamA}
	\frac{ \diam(G,g_I|_{\CC_{P,k}}^* )} {\sigma_k(A)} 
	\leq \diam(G,g_A) \leq 
	\frac{\diam(G,\HH_{P,k},g_I|_{\HH_{P,k}} )}{\sigma_k(A)}.
\end{equation}
\end{proposition}

\begin{proof}
	We fix $P$ in $\Ot(m)$ sorting $A$. 
	We abbreviate $\sigma_j=\sigma_j(A)$ for all $j$. 
	Similarly as in \eqref{eq2:A^tAprimero}, we have that 
	\begin{align*}
		AA^t
		&\leq P\diag(\sigma_1^2,\dots,\sigma_{k-1}^2,\sigma_{k}^2, \dots,\sigma_k^2)P^t
		=\sigma_k^2\; P\diag((\tfrac{\sigma_{1}}{\sigma_k})^2, \dots,(\tfrac{\sigma_{k-1}}{\sigma_k})^2,1,\dots,1) P^t.
	\end{align*}
	We set $B_1=P\diag(\tfrac{\sigma_{1}}{\sigma_k}, \dots,\tfrac{\sigma_{k-1}}{\sigma_k},1,\dots,1)$.
	Since $AA^t\leq \sigma_k^2 \, B_1B_1^t$, Lemma~\ref{lem3:diam-inequality} implies that 
	$$
	\diag(G,g_A)
	\geq \diag(G,g_{\sigma_k B_1}) 
	= \sigma_k^{-1}\diag(G,g_{B_1}).
	$$ 
Hence, the inequality at the left-hand side in \eqref{eq3:diamA} follows since $\diag(G,g_{B_1})\geq \diam(G,g_I|_{\CC_{P,k}}^* )$ by Lemma~\ref{lem2:sing-diam-restriccionRiemanniano}. 

We now establish the inequality at the right in \eqref{eq3:diamA}. 
Similarly as above, by setting $B_2=P\diag(1,\dots,1,\tfrac{\sigma_{k+1}}{\sigma_k}, \dots,\tfrac{\sigma_{m}}{\sigma_k})$, one has $AA^t\geq \sigma_k^2\, BB^t$, thus $\diam(G,g_A)\leq \sigma_k^{-1}\diam(G,g_{B_2})$.
The assertion follows since $\diam(G,g_{B_2})\leq \diam(G,\HH_{P,k},g_I|_{\HH_{P,k}} )$ by Lemma~\ref{lem2:sub-diam-restriccionRiemanniano}. 
\end{proof}

\begin{remark}
	Some words of caution about \eqref{eq3:diamA} are necessary at this point. 
	Unlike in \eqref{eq3:diam-simple-estimates}, the coefficients in the extremes depend on $A$ (more precisely on $P$). 
	Moreover, the inequality at the left (resp.\ right) hand side is useless when $\diam(G,g_I|_{\CC_{P,k}}^*)=0$ (resp.\ $\diam(G,\HH_{P,k}, g_I|_{\HH_{P,k}} )=\infty $).
\end{remark}

\subsection{Optimal indices for the diameter} \label{subsec:diam-optimal}

It is desirable to improve the estimates in \eqref{eq3:diam-simple-estimates} by replacing $\sigma_m(A)$ (resp.\ $\sigma_1(A)$) at the right-hand (resp.\ left-hand) side by $\sigma_k(A)$ with $k$ as small (resp.\ large) as possible (see Notation~\ref{not2:sigma_j(A)} for the definition of $\sigma_k(A)$).

We first study the left-hand side case. 
In what follows, for any subset $S$ of $\fg$, we set $S^\perp = \{X\in\fg: g_I(X,Y)=0\text{ for all }Y\in S\}$.

\begin{theorem}\label{thm3:optimal-index-lower-bound}
Assume that $G$ is non-abelian. 
There is a positive real number $C_1$ depending on $G$ and $g_I$ such that 
\begin{equation}\label{eq3:optimallower}
\frac{C_1}{\sigma_{2}(A)}\leq \diam(G,g_A)
\qquad\text{for all $A\in\GL(m,\R)$. }
\end{equation}
Furthermore, 
\begin{equation}
\label{eq3:diam-optimal-non-true-lowerbound}
\inf_{A\in\GL(m,\R) }\diam(G,g_A) \;\sigma_{4}(A)=0.
\end{equation}
Consequently, the largest index $k$ satisfying that $\inf_{A\in\GL(m,\R) }\diam(G,g_A)\, \sigma_{k}(A)>0$ is $2$ or $3$. 
\end{theorem}

\begin{proof}
Let $T$ be a maximal torus in $G$ with Lie algebra $\ft$. 
We have that $\ft \neq \fg$ by assumption. 
We claim that \eqref{eq3:optimallower} holds by setting $C_1= \diam(G/T, g_I|_{\ft^{\perp}})$. 

We fix $A\in \GL(m,\R)$.
Let $P$ be any matrix in $\Ot(m)$ sorting $A$.
Write $\mathcal C=\mathcal C_{P,2} = \Span_\R\{ X_2(P),\dots,X_m(P) \}$.
By Proposition~\ref{prop3:diam-k}, we have that 
$$
\diam(G,g_A)\geq \frac{ \diam(G,g_I|_{\mathcal C}^*) }{\sigma_2(A)}.
$$

The subspace $\Span_\R\{X_1(P)\}$ of $\fg$ is of course an abelian subalgebra of $\fg$.
Since any two maximal abelian subalgebras of $\fg$ are conjugate via $\Ad(G)$ (see for instance \cite[Thm.~4.34]{Knapp-book-beyond}), there is $a\in G$ such that $X_1(P) \in \Ad_a(\ft)$. 
Clearly, $\fp:= (\Ad_a(\ft))^{\perp} \subset \mathcal C$. 
Lemma~\ref{lem2:sing-diam-monotonicity} gives
\begin{align*}
\diam(G,g_I|_{\mathcal C}^*) \geq \diam(G,g_I|_{\fp}^*)=\diam(G,g_I|_{\ft^{\perp}}^*).
\end{align*}
The last step follows since $(G,g_I|_{\fp}^*)$ and $(G,g_I|_{\ft^{\perp}}^*)$ are isometric because $g_I$ is $\Ad_a$-invariant.
Hence, \eqref{eq3:optimallower} follows by noting that $\diam(G,g_I|_{\ft^{\perp}}^*)= \diam(G/T,g_I|_{\ft^{\perp}})$ by Lemma~\ref{lem2:diam(G/d)=diam(G/H)}. 

We next establish \eqref{eq3:diam-optimal-non-true-lowerbound}. 
It is well known that $\fg=\fg_{\text{ss}}\oplus Z(\fg)$, where $Z(\fg)$ is the center of $\fg$ and $\fg_{\text{ss}}=[\fg,\fg]$ is semisimple (see e.g.\ \cite[Cor.~4.25]{Knapp-book-beyond}). 
Furthermore, $G$ is the commuting product $G=G_{\textrm{ss}}(Z_G)_0$ between the analytic subgroup $G_{\textrm{ss}}$ of $G$ corresponding to $\fg_{\textrm{ss}}$ and the connected component $(Z_G)_0$ of the center $Z_G$ of $G$, which coincides with the analytic subgroup of $G$ with Lie algebra $Z(\fg)$ (see e.g.\ \cite[Thm.~4.29]{Knapp-book-beyond}). 
The subgroups $G_{\textrm{ss}}$ and $(Z_G)_0$ are closed in $G$. 

Since $\fg_{\text{ss}}$ is semisimple, there is a bracket-generating set of $\fg_{\text{ss}}$ having two elements (see \cite[Thm.~ 6]{Kuranishi}).
That is, there are $Y_1,Y_2\in \fg_{\text{ss}}$ such that no any proper subalgebra of $\fg_{\text{ss}}$ contains them simultaneously. 
Let $Y_3\in Z(\fg)$ satisfying that $L:=\{\exp(tY_3):t\in\R\}$ is a dense subgroup of $(Z_G)_0$.
Let $P\in\Ot(m)$ satisfying that $\Span_\R\{X_1(P),X_2(P)\} =\Span_\R\{Y_1,Y_2\}$ and $\Span_\R\{X_1(P),X_2(P),X_3(P)\} =\Span_\R\{Y_1,Y_2,Y_3\}$. 
For any $s>0$, let 
$$
D_s = \diag( s,s,s, \underbrace{1,\dots,1}_{(m-3)\text{-times}}).
$$ 
One clearly has $\sigma_{4}(PD_s) = 1$ for all $s\geq1$. 
We claim that 
$$
\lim_{s\to\infty} \;\diam(G,g_{PD_s})\, \sigma_{4}(PD_s)=
\lim_{s\to\infty} \;\diam(G,g_{PD_s})= 0. 
$$

It is sufficient to prove that the distance between $e$ and an arbitrary element $a\in G$ goes to zero as $s\to\infty$. 
Let $\varepsilon>0$. 
Write $a=bz$ with $b\in G_{\textrm{ss}}$ and $z\in (Z_G)_0$. 
Since $L$ is dense in $(Z_G)_0$, there is $t_0>0$ such that $\dist_{(G,g_I)}(z,z_0)<\frac{\varepsilon}{2}$, where $z_0=\exp(t_0Y_3)$. 
Since $\Span_\R\{Y_1,Y_2\}$ is bracket generating in $\fg_{\textrm{ss}}$, Theorem~\ref{thm2:Chow} yields there is a smooth curve $\gamma:[0,1]\to G_{\textrm{ss}}$ with $\gamma(0)=e$, $\gamma(1)=b$, and $\gamma'(t)\in \Span_\R\{Y_1,Y_2\}$ for all $t$. 
Thus, the curve $\eta:[0,1]\to G$ given by $\eta(t)=\gamma(t)\exp(t\,t_0Y_3)$ satisfies $\eta(0)=e$, $\eta(1)=bz=a$, and $\eta'(t)\in \Span_\R\{Y_1,Y_2,Y_3\}$ for all $t$. 
Hence,
\begin{align*}
\dist_{(G,g_{PD_s})} (e,a) 
&\leq  \dist_{(G,g_{PD_s})} (bz,bz_0)+ \dist_{(G,g_{PD_s})} (e,bz_0) 
\\ 
&\leq  \dist_{(G,g_{I})} (bz,bz_0)+ \int_{0}^1 g_{PD_s}(\eta'(t),\eta'(t))^{1/2}\, dt
\\
&\leq  \frac{\varepsilon}{2}+ \frac{1}{s^2} \int_{0}^1 g_{I}(\eta'(t),\eta'(t))^{1/2}\, dt<\varepsilon
\end{align*}
for $s$ sufficiently large. 
Since this holds for all $\varepsilon>0$, we conclude that $\dist_{(G,g_{PD_s})} (e,a) =0$ as required. 
\end{proof}

\begin{question}
Is $2$ the largest index $k$ satisfying that $\inf_{A\in\GL(m,\R) }\diam(G,g_A)\, \sigma_{k}(A)>0$? 
\end{question}

The proof of Theorem~\ref{thm3:optimal-index-lower-bound} shows that the answer of this question is $2$ when $G$ is semisimple. 

\begin{remark}\label{rem3:flat-lower}
Assume that $G$ is abelian. 
We claim that the left-hand side of \eqref{eq3:diam-simple-estimates} has the optimal index $k=1$, that is,
\begin{equation}
\inf_{A\in\GL(m,\R) }\diam(G,g_A)\, \sigma_{2}(A)=0. 
\end{equation}

Write $G=S^1\times\dots\times S^1$ ($m$-times), and let $X_1,\dots,X_m$ denote the basis of $\mathfrak g$ such that $\exp(tX_i)$ is the closed curve in $G$ with period $1$ staying in the $i$th $S^1$-component. 
Let $g_I$ be the left-invariant metric on $G$ satisfying that this basis is orthonormal. 

Let $P\in\GL(m,\R)$ satisfying that $L:=\{\exp(tX_1(P)):t\in\R\}$ is dense in $G$. 
Now, taking $D_s=\diag(s,1,\dots,1)$ for any $s>0$, we have that $\sigma_2(PD_s)=1$ for all $s>1$ and 
\begin{equation}
\lim_{s\to\infty} \;\diam(G,g_{PD_s})\, \sigma_{2}(PD_s)=
\lim_{s\to\infty} \;\diam(G,g_{PD_s})= 0
\end{equation}
since $L$ is dense and the distance between two fixed points in $L$ goes to $0$ when $s\to\infty$. 
\end{remark}

We next show that the optimal index for the right-hand side in \eqref{eq3:diam-simple-estimates} is given by 
\begin{equation}\label{eq3:k_max}
k_{\max} =k_{\max}(G):=1+\max_{H} \,\dim H,
\end{equation}
where $H$ runs over the closed subgroups of $G$ with Lie algebra $\fh\neq\fg$, provided an external fact from metric geometry holds (Condition~\ref{claim}).  
We first state this fact.

For any index $k$, let us denote by $\op{Gr}_\fg(k)$ the space of $k$-dimensional subspaces of $\fg$, which has a structure of symmetric space known as a (real) Grassmannian space.
We will only use the corresponding underlying topology on it, which in fact makes it compact.   

\begin{lemma}\label{lem3:k_max}
If $\HH\in\op{Gr}_{\fg}(k)$ for any $k\geq k_{\max}$, then $\HH$ is bracket-generating. 
\end{lemma}

\begin{proof}
	If $G$ is not semisimple, the assertion is obvious since $k_{\max}=m=\dim G$ because there are closed subgroups of $G$ of codimension one. 
	
	We assume that $G$ is semisimple. 
	Let $\HH$ be a $k$-dimensional subspace of $\fg$ with $k\geq k_{\max}$. 
	Let $\fh$ be the smallest subalgebra of $\fg$ containing $\HH$. 
	The corresponding analytic subgroup $H$ of $G$ with Lie algebra $\fh$ is dense in $G$. 
	In fact, if $\overline H$ is a proper subgroup of $G$, then its dimension is strictly less than $k_{\max}$ by \eqref{eq3:k_max}, which is a contradiction since $k=\dim\HH\leq \dim \fh\leq \dim \overline H$. 
	However, since there are no dense proper connected Lie subgroups in a compact semisimple Lie group (cf.\ \cite{MaciasVirgos}), we obtain that $H=G$ and consequently, $\HH$ is bracket-generating in $\fg$. 
\end{proof}

For $k\geq k_{\max}$, Lemma~\ref{lem3:k_max} implies $\diam(G,\HH,g_I|_{\HH})<\infty$ for all $\HH\in\op{Gr}_{\fg}(k)$ by Theorem~\ref{thm2:Chow}.

\begin{condition}\label{claim}
If $G$ is semisimple and $k\geq k_{\max}$, then the map $\Upsilon_k: \op{Gr}_\fg(k) \to \R_{>0}$ given by $\Upsilon_k(\HH) = \diam(G,\HH,g_I|_{\HH})$ is continuous. 
\end{condition}

It is sufficient to prove that the map $\HH\mapsto (G,\HH,g_I|_{\HH})$, whose target is the space of metric spaces endowed with the Gromov-Hausdorff distance, is continuous. 
In fact, $\Upsilon_k$ is the composition of this map with the diameter function, which is continuous (see \cite[Ex.~7.3.14]{BuragoBuragoIvanov-book}).%
%%%%%%%%%%%%
% new part %
\footnote{
Update: the validity of Condition~\ref{claim} follows by Theorem 1.6 (see Remark 4.4) in \cite{AntonelliLeDonneGolo}.
Consequently, the proof of Theorem~\ref{thm3:optimal-index-upper-bound} is complete, as well as for Theorem~\ref{thm:main0}. }
%%%%%%%%%%%%

\begin{theorem}\label{thm3:optimal-index-upper-bound}
We have that
\begin{equation}
	\label{eq3:diam-optimal-non-true-upperbound}
	\sup_{A\in\GL(m,\R) }\diam(G,g_A) \;\sigma_{k_{\max}-1}(A)=\infty.
\end{equation}
Consequently, if $G$ is not semisimple, then $k_{\max}=m$ is the smallest index $k$ satisfying that $\sup_{A\in\GL(m,\R) }\diam(G,g_A)\, \sigma_{k}(A)<\infty$.

When $G$ is semisimple, if Condition~\ref{claim} holds, then there is a positive real number $C_2$ depending on $G$ and $g_I$ such that 
\begin{equation}\label{eq3:optimalupper}
\diam(G,g_A)\leq \frac{C_2}{\sigma_{k_{\max}}(A)},
\qquad\text{for all $A\in\GL(m,\R)$}.
\end{equation}
Consequently, $k_{\max}$ is the smallest index $k$ satisfying that $\sup_{A\in\GL(m,\R) }\diam(G,g_A)\, \sigma_{k}(A)<\infty$.
\end{theorem}

\begin{proof}
	We first show \eqref{eq3:diam-optimal-non-true-upperbound}.
	Let $H$ be a proper closed subgroup of $G$ with Lie algebra $\fh\neq\fg$. 
	The dimension $n$ of $H$ clearly satisfies $n\leq k_{\max}-1$.
	There is $P\in\Ot(m)$ such that $\fh = \Span_\R\{X_1(P),\dots, X_n(P)\}$. 
	For $s>0$, let 
	$$
	D_s=\diag(\underbrace{1,\dots,1}_{n\text{-times}}, \underbrace{s,\dots,s}_{(m-n)\text{-times}}).
	$$
	One clearly has that $\sigma_n(PD_s)=1$ for all $s\leq 1$. 
	We claim that 
	$$
	\lim_{s\to0^+} \;\diam(G,g_{PD_s})\, \sigma_n(PD_s)=
	\lim_{s\to0^+} \;\diam(G,g_{PD_s}) =\infty.
	$$ 
	In fact, we will show that the distance between $e$ and any point $a\notin H$ goes to infinity when $s\to 0$.  
	Let $\gamma:[0,1]\to G$ be any smooth curve with $\gamma(0)=e$ and $\gamma(1)=a$. 
	We write $\gamma'(t)=\gamma'_{\fh}(t)+\gamma'_\fp(t)$ with $\gamma_\fh'(t) \in \fh$ and $\gamma_\fp'(t) \in\fp:= \fh^{\perp_{g_I}}= \{X\in \fg: g_I(X,\fg)=0\}$. 
	We note that $\gamma_\fp'\equiv 0$ is not possible since in this case $\gamma(t)$ will stay in $H$ for all $t$. 
	It follows that 
	\begin{align*}
		\op{length}_{PD_s}(\gamma) 
		&= \int_{0}^1 g_{PD_s}(\gamma'(t),\gamma'(t))^{1/2}\, dt
\geq \frac{1}{s}\, \int_0^1 g_{I}(\gamma_\fp'(t),\gamma_\fp'(t))^{1/2}\, dt,
	\end{align*}
	which goes to infinity as $s\to0^+$ since $\int_0^1 g_{I}(\gamma_\fp'(t),\gamma_\fp'(t))^{1/2}\, dt>0$. 
	The proof of \eqref{eq3:diam-optimal-non-true-upperbound} is complete by taking $H$ of dimension $k_{\max}-1$, which obviously exists by \eqref{eq3:k_max}. 
	
If $G$ is not semisimple, then $G$ contains closed subgroups of codimension one, thus $k_{\max}=m$. 
Therefore, the right hand side of \eqref{eq3:diam-simple-estimates} gives already the optimal index. 
	
We next prove \eqref{eq3:optimalupper} under the assumptions that $G$ is semisimple and Condition~\ref{claim} holds. 
We set $C_2:= \sup_{\HH\in\op{Gr}_\fg(k)} \Upsilon_{k_{\max}}(\HH)$, which is finite since $\Upsilon_{k_{\max}}: \op{Gr}_\fg(k_{\max})\to\R_{>0}$ is continuous (by Condition~\ref{claim}) and $\op{Gr}_\fg(k)$ is compact. 
Now, \eqref{eq3:optimalupper} follows immediately from 
Proposition~\ref{prop3:diam-k}. 
\end{proof}

\begin{remark}\label{rem3:SU(2)}
For $G=\SU(2)$, the existences of $C_1$ and $C_2$ satisfying \eqref{eq3:optimallower} and \eqref{eq3:optimalupper} were established by Eldredge, Gordina and Saloff-Coste (see \cite[Lem.~7.1]{EldredgeGordinaSaloff18}). 
Furthermore, \cite[Cor.~4.4]{Lauret-SpecSU(2)} gives explicit values for them, when $G$ is $\SU(2)$ or $\SO(3)$, namely
	\begin{align*}
		\frac{\pi/2}{\sigma_2(A)} &\leq \diam(\SU(2),g_A)\leq \frac{\pi}{\sigma_2(A)}, 
		&
		\frac{\pi/2}{\sigma_2(A)} &\leq \diam(\SO(3),g_A)\leq \frac{\sqrt{3}\pi/2}{\sigma_2(A)},
	\end{align*}
for all $A\in \GL(3,\R)$.
\end{remark}

\subsection{Diameter estimate for a restricted subclass}\label{subsec:diam-restricted}
We now return to the discussion of Eldredge, Gordina, and Saloff-Coste's conjecture. 
The next section will contain analogous result as Theorems~\ref{thm3:optimal-index-lower-bound} and \ref{thm3:optimal-index-upper-bound} having optimal choices for the indices concerning the first eigenvalue $\lambda_1(G,g_A)$ of the Laplace--Beltrami operator associated to $(G,g_A)$. 
However, we will see in Section~\ref{sec:upperbounds} that these estimates are not sufficient to prove the EGS conjecture beside for $\SU(2)$, or $\SO(3)$, which are the only non-abelian known cases so far. 
The next goal is to refine the estimates in  Proposition~\ref{prop3:diam-k} to obtain uniform estimates valid for a large subclass of left-invariant metrics on $G$.

\begin{definition} \label{def3:bracket-generating-indexA} 
We associate to an element $P\in\Ot(m)$ the following objects:
\begin{itemize}
\item the \emph{bracket-generating index} $\ell(P)$ given by the smallest positive integer $k$ satisfying that $\{X_1(P),\dots, X_k(P)\}$ is bracket generating,

\item $\fh_P$ denotes the Lie subalgebra of $\fg$ generated by $\HH_{P,\ell(P)-1} = \{X_1(P),\dots,X_{\ell(P)-1}(P)\}$, 

\item $H_P$ denotes the only connected subgroup of $G$ with Lie algebra $\fh_P$, 

\item $\bar H_P$ denotes the closure of $H_P$ which is a closed subgroup of $G$, 

\item $\bar \fh_P$ denotes the Lie algebra of $\bar H_P$,

\item $\bar\fh_P^\perp$ denotes the orthogonal complement of $\bar\fh_P$ in $\fg$ with respect to $g_I$,

\item the subclass of left-invariant metrics on $G$ given by 
\begin{equation}
\mathcal M^G(P):= \{g_{PQD}: Q\in\mathcal O(m,\ell(P)),\, D\in\mathcal D(m)\},
\end{equation}
where 
\begin{equation}\label{eq3:D(m)O(m,k)}
\begin{aligned}
	\mathcal D(m) &=\{\diag(\sigma_1,\dots,\sigma_m) \in \GL(m,\R): \sigma_1\geq\dots\geq \sigma_m>0\},
	\\
	\mathcal O(m,k)&= \left\{
	\left(\begin{smallmatrix}
		Q_1\\ &1\\ &&Q_2
	\end{smallmatrix}\right)
	: Q_1\in\Ot(k-1),\, Q_2\in \Ot(m-k)\right \}\subset \Ot(m). 
\end{aligned} 
\end{equation}
\end{itemize}
\end{definition}

It is important to note that the objects introduced in Definition~\ref{def3:bracket-generating-indexA} depend on the choices of the $\Ad(G)$-invariant inner product $g_I(\cdot,\cdot)$ on $\mathfrak g$ and its orthonormal basis $\mathcal B$ in Subsection~\ref{subsec:left-invmetrics}. 
(This situation was anticipated at the beginning of Subsection~\ref{subsec:left-invmetrics}.)

\begin{theorem}\label{thm3:diam(g_PQD)}
Let $P\in\Ot(m)$ and set $k=\ell(P)$. 
We have that 
\begin{equation}\label{eq3:diam(g_PQD)geq}
\diam(G,g_{A}) \leq 
\frac{\diam(G,\HH_{P,k}, g_I|_{\HH_{P,k}})} {\sigma_{k}(A)} <\infty
\qquad \text{for every } g_A\in \mathcal M^G(P).
\end{equation}
Furthermore, if $\bar H_P\neq G$ 
(e.g.\ if $G$ is semisimple), then 
\begin{equation}\label{eq3:diam(g_PQD)leq}
\diam(G,g_{A}) \geq \frac{\diam(G/\bar H_P,g_I|_{\bar\fh_P^\perp})} {\sigma_{k}(A)} >0
\qquad \text{for every }g_A\in \mathcal M^G(P).
\end{equation}
\end{theorem}

\begin{proof}
Let $Q\in\mathcal O(m,k)$. 
We have that $X_j(PQ) = \sum_{i=1}^m (PQ)_{i,j}X_i = \sum_{i=1}^m \sum_{l=1}^m P_{i,l}Q_{l,j}X_i = \sum_{l=1}^m Q_{l,j} X_l(P)$, thus
\begin{align*}
X_j(PQ) 
=
\begin{cases}
\sum\limits_{l=1}^{k-1} Q_{l,j}X_l(P)&\text{ if }j<k,\\[4mm]
X_k(P)&\text{ if }j=k,\\[2mm]
\sum\limits_{l=k+1}^{m} Q_{l,j}X_l(P)&\text{ if }j>k.
\end{cases}
\end{align*}
It follows that $\HH_{PQ,k}= \HH_{P,k}$, $\CC_{PQ,k}=\CC_{P,k}$, and $\ell(PQ)=\ell(P)=k$. 
Hence, for any $D\in\mathcal D(m)$, Proposition~\ref{prop3:diam-k} implies that 
\begin{equation*}
\frac{ \diam(G,g_I|_{\CC_{P,k}}^* )} {\sigma_k(PQD)} 
\leq \diam(G,g_{PQD}) \leq 
\frac{\diam(G,\HH_{P,k},g_I|_{\HH_{P,k}} )}{\sigma_k(PQD)}.
\end{equation*}
Note that $\sigma_k(PQD)=\sigma_k(D)$ for all $D\in\mathcal D(m)$. 
From Chow--Rashevskii Theorem (Theorem~\ref{thm2:Chow}), it follows that $\diam(G,\HH_{P,k},g_I|_{\HH_{P,k}}) <\infty$, showing \eqref{eq3:diam(g_PQD)geq}.

It remains to show that $\diam(G,g_I|_{\CC_{P,k}}^* ) \geq \diam(G/\bar H_P,g_I|_{\bar\fh_P^\perp})>0$. 
One has that $\CC_{P,k}\supset \bar \fh_P^\perp\neq 0$ since $\Span_\R\{X_1(P),\dots,X_{k-1}(P)\} \subset \fh_P\subset \bar \fh_P \neq \fg$. 
Therefore $\diam(G,g_I|_{\CC_{P,k}}^* )\geq \diam(G,g_I|_{\bar\fh_P^\perp}^* )$ by Lemma~\ref{lem2:sing-diam-monotonicity}. 
We conclude that $\diam(G,g_I|_{\bar\fh_P^\perp}^* ) = \diam(G/\bar H_P,g_I|_{\bar\fh_P^\perp})$ by Lemma~\ref{lem2:diam(G/d)=diam(G/H)}. 
That $\diam(G/\bar H_P,g_I|_{\bar\fh_P^\perp})>0$ follows from $\bar \fh_P\neq 0$. 
\end{proof}

We next give a new version of \eqref{eq3:diam(g_PQD)leq} for $G$ semisimple.
We will replace ${\diam(G/\bar H_P,g_I|_{\bar\fh_P^\perp})}$ by a constant independent of $P$.
However, the inequality is still valid to the restricted subclass $\mathcal M^G(P)$ which does depend on $P$.  
We first need some tools from Lie theory.
The next lemma is well known, but we include a proof for completeness, which was provided by the mathoverflow user Ycor~\cite{Ycor}.

\begin{lemma}\label{lem3:max-subalgebras}
When $\fg$ is semisimple, there are finitely many maximal subalgebras in $\fg$ up to conjugation.
\end{lemma}

\begin{proof}
It is well known that there are finitely many semisimple subalgebras in $\fg$ up to conjugation (cf.\ \cite[Prop.~12.1]{Richardson}). 
The assertion of the lemma follows since 
any maximal subalgebra of $\fg$, if not abelian, is of the form $\fh\oplus \{X\in\fg: [X,\fh]=0\}$ for some semisimple subalgebra $\fh$ of $\fg$. 
Note that if a maximal subalgebra of $\fg$ is abelian, it is the Lie algebra of a maximal torus which is unique up to conjugation. 
\end{proof}

The next remark translates the previous result to a statement which will be useful later. 

\begin{remark}\label{rem3:max-subalgebras}
Assume that $\fg$ is semisimple. 
By Lemma~\ref{lem3:max-subalgebras}, there are $\fh_1,\dots,\fh_r$ proper Lie subalgebras of $\fg$ such that
\begin{itemize}
	\item for any Lie subalgebra $\fa$ of $\fg$, there is $a\in G$ such that $\fa\subset \Ad_a(\fh_i)$ for some $i$;
	
	\item if a Lie subalgebra $\fa$ of $\fg$ contains properly $\fh_i$ for some $i$, then $\fa=\fg$.
\end{itemize}

For each $i$, let $H_i$ denote the only connected Lie subgroup of $G$ with Lie algebra $\fh_i$. 
It turns out that $H_i$ is closed in $G$. 
In fact, $\bar H_i=H_i$ or $\bar H_i=G$ since the Lie algebra $\bar\fh_i$ of $\bar H_i$ satisfies $\fh_i\subset \bar\fh_i\subset\fg$, but $\bar H_i=G$ is not possible because there are no dense proper connected Lie subgroups in a compact semisimple Lie group (cf.\ \cite{MaciasVirgos}). 
Moreover, the Riemannian manifold $(G/H_i,g_I|_{\fh_i^\perp})$ does not depend on the choice of $H_i$ since $g_I$ is invariant by conjugation. 
\end{remark}

\begin{corollary}\label{cor3:diam-inf-universal}
Let $G$ be a compact connected semisimple Lie group. 
Under the notation introduced in Remark~\ref{rem3:max-subalgebras}, we set
\begin{equation}
C_1'=\min_{1\leq i\leq r} \, \diam(G/H_i,g_I|_{\fh_i^{\perp}}),
\end{equation}
which is positive and depends only on $G$ and $g_I$. 
Then, for any $P\in\Ot(m)$, 
\begin{equation}\label{eq3:diamP-geq-independent}
\diam(G,g_{A}) \geq \frac{C_1'}{\sigma_{\ell(P)}(A)}
\qquad\text{for every } A\in\mathcal M^G(P).
\end{equation}
\end{corollary}

\begin{proof}
Fix any $P\in\Ot(m)$. 
Since $G$ is semisimple, $\bar\fh_P\neq\fg$, thus $\bar\fh_P\subset \Ad_a(\fh_i)$ for some $i$ and $a\in G$ (see Remark~\ref{rem3:max-subalgebras}). 
We have that
\begin{multline*}
\diam(G/\bar H_P,g_I|_{\bar\fh_P^\perp}) = \diam (G, g_I|_{\bar\fh_P^\perp}^*) \geq 
\diam (G, g_I|_{\Ad_a(\fh_i^\perp)}^*) 
\\ =\diam (G,g_I|_{\fh_i^\perp}^*) = \diam (G/H_i,g_I|_{\fh_i^\perp})
\end{multline*}
In fact, the first and last equality follow from Lemma~\ref{lem2:diam(G/d)=diam(G/H)}, the second equality follows since $(G,g_I|_{\Ad_a(\fp_i)}^*)$ and $(G,g_I|_{\fp_i}^*)$ are isometric because $g_I$ is $\Ad_a$-invariant, and the inequality follows from Lemma~\ref{lem2:sing-diam-monotonicity} since $\bar\fh_P^\perp\supset \Ad_a(\fh_i^\perp)$.
\end{proof}

\begin{remark}
There should not exist an upper bound for $\diam(G,g_A)$ for all $g_A\in\mathcal M^G(P)$ independent on $P$ analogous to \eqref{eq3:diamP-geq-independent}. 
This is because $\diam(G,\HH_{P,\ell(P)},g_I|_{\HH_{P,\ell(P)}})$ may not be bounded by above uniformly for all $P\in\Ot(m)$. 
For instance, the author expects that, if a sequence $P_j\in\Ot(m)$ for $j\in\N$ converging to $P_0\in\Ot(m)$ satisfies that $\ell(P_j)$ is constant and $\ell(P_j)<\ell(P_0)$, then
$$
\lim_{j\to\infty} \diam(G,\HH_{P_j,\ell(P_j)},g_I|_{\HH_{P_j,\ell(P_j)}}) = \diam(G,\HH_{P_0,\ell(P_0)},g_I|_{\HH_{P_0,\ell(P_0)}})=\infty. 
$$
\end{remark}

\section{Eigenvalue estimates} \label{sec:eigenvalues}
We continue assuming that $G$ is a compact connected Lie group of dimension $m$. 
This section considers estimates for the first non-zero eigenvalue of the Laplace--Beltrami operator associated to $(G,g_A)$ (for $A\in\GL(m,\R)$) in terms of the eigenvalues of $AA^t$.
We will proceed analogously to the previous section. 

We will use the correspondence $\GL(m,\R)\ni A\mapsto g_A \in \mathcal M^G$ introduced in Subsection~\ref{subsec:left-invmetrics}, as well as the abstract description of the spectrum of the Laplace--Beltrmai operator $\Delta_A$ associated to $(G,g_A)$ in Subsection~\ref{subsec:spectraleft-invmetrics}.

\subsection{Simple estimates for the first eigenvalue}
We have seen in \eqref{eq2:A^tAprimero}, for any $A\in\GL(m,\R)$, that $\sigma_m(A)^2I\leq AA^t\leq \sigma_1(A)^2I$. 
Recall from Notation~\ref{not2:sigma_j(A)} that $\sigma_1(A)^2$ and $\sigma_m(A)^2$ stand for the largest and smallest eigenvalues of $AA^t$ respectively. 
The estimates
\begin{equation}\label{eq4:lambda1-simple-estimates}
\lambda_1(G,g_I)\, \sigma_m(A)^2
\leq \lambda_1(G,g_A)  \leq 
\lambda_1(G,g_I)\, \sigma_1(A)^2
\end{equation}
follow immediately form the next result.

\begin{lemma} \label{lem4:lambda1-inequality}
Let $A,B\in \GL(m,\R)$ satisfying $AA^t\leq BB^t$. 
	Then $\pi(-C_A)\leq \pi(-C_B)$ for every finite dimensional unitary representation $\pi$ of $G$. 
	Moreover, 
	$$
	\lambda_1(G,g_A)\leq \lambda_1(G,g_B).
	$$
\end{lemma}

\begin{proof}
	Let $(\pi,V_\pi)$ be any finite dimensional unitary representation of $G$. 
	Since $AA^t\leq BB^t$, there is $P\in\Ot(m)$ such that $0\leq BB^t-AA^t= PD^2P^t$ for some $D=\diag(d_1,\dots,d_m)$ with $d_j\in\R$ for all $j$. 
	This implies that $BB^t=AA^t+PD^2P^t$, thus 
	$$
	C_B = \sum_{i,j=1}^m (BB^t)_{i,j}\, X_iX_j
	=\sum_{i,j=1}^m \big((AA^t)_{i,j} + (PD^2P^t)_{i,j}\big) \, X_iX_j
	= C_A+C_{PD}. 
	$$
Consequently, $\pi(-C_B)=\pi(-C_A)+\pi(-C_{PD})\geq \pi(-C_A)$ since $\pi(-C_{PD})\geq0$.
	In fact, Lemma~\ref{lem2:C_AB} gives 
	$
	C_{PD} 
	= \sum_{i,j=1}^m (DD^t)_{i,j} \, X_i(P)X_j(P) 
	= \sum_{j=1}^m d_j^2 \, X_j(P)^2,
	$ 
	thus 
	$
	\pi(-C_{PD}) 
	= \sum_{j=1}^m d_j^2 \, \pi(-X_j(P)^2) 
	= -\sum_{j=1}^m d_j^2 \, \pi(X_j(P))^2 \geq0. 
	$

We now show the second assertion. 
For any $\pi\in \widehat G$, we have seen that $\pi(-C_A)\leq \pi(-C_B)$, in particular, $\lambda_{\min}(\pi(-C_A))\leq  \lambda_{\min}(\pi(-C_B))$. 
	Hence, \eqref{eq2:lambda1(G,g_A)} immediately implies that $\lambda_1(G,g_A)\leq \lambda_1(G,g_B)$. 
\end{proof}

\begin{remark}\label{rem4:Urakawa}
Notice the right hand side of \eqref{eq4:lambda1-simple-estimates} improves the following  estimate by Urakawa (see \cite[Thm.~ 3]{Urakawa79}):
\begin{equation}\label{eq4:Urakawa}
\lambda_1(G,g_A)\leq \lambda_1(G,g_I)\, \Tr(AA^t)
\qquad\text{for all $A\in\GL(m,\R)$.}
\end{equation}
In fact, $\Tr(AA^t)=\sum_{j=1}^m \sigma_j(A)^2> \sigma_1(A)^2$ for any $A\in\GL(m,\R)$. 
Moreover, the estimates in \eqref{eq4:lambda1-simple-estimates} are sharp in the sense that they are attained when $A$ is a positive multiple of $I$. 
\end{remark}

In the next subsections we look for estimates as in \eqref{eq4:lambda1-simple-estimates} with the index $m$ (resp.\ $1$) at the left-hand side (reps.\ right-hand side) replaced by an index $k$ as small (resp.\ large) as possible.

\subsection{Main tool for the first eigenvalue}

Proposition~\ref{prop4:lambda1-k} will be the main tool in the rest of the section. 
We need some preliminaries to state it beside those in Subsection~\ref{subsec:spec-non-Riemannian}. 

For any $P\in\Ot(m)$ and $1\leq k \leq m$, we set 
\begin{equation}
C_{P,k}^\infty(G) =\{f\in C^\infty(G): X\cdot f=0 \;\text{for all } X\in\HH_{P,k-1}\}.
\end{equation} 
We recall from Notation~\ref{not2:HH_P^kCC_P^k} that $\HH_{P,k-1}=\Span_\R\{X_1(P),\dots,X_{k-1}(P)\}$. 
By \eqref{eq2:PeterWeyl}, the closure of $C_{P,k}^\infty(G)$ in the Hilbert space $L^2(G)$ is given by
\begin{equation}\label{eq4:closureC_Pk^infty}
\op{closure}(C_{P,k}^\infty(G)) = \bigoplus_{\pi\in\widehat G} V_\pi^{\HH_{P,k-1}} \otimes V_\pi^*,
\end{equation}
where $V_\pi^{\HH_{P,k-1}}=\{ v\in V_\pi: \pi(X)\cdot v=0\text{ for all }X\in\HH_{P,k-1}\}$. 
In particular, the Laplace--Beltrami operator $\Delta_I$ of $(G,g_I)$ preserves $C_{P,k}^\infty(G)$. 
Whenever $C_{P,k}^\infty(G)$ has dimension strictly greater than one, we denote by  $\lambda_1(\Delta_I|_{C_{P,k}^\infty(G)})$ the smallest positive eigenvalue of $\Delta_I|_{C_{P,k}^\infty(G)}$.
Furthermore, every eigenfunction of $\Delta_I|_{C_{P,k}^\infty(G)}$ is of the form $f_{v\otimes \varphi}$ with $v\in V_\pi^{\HH_{P,k-1}}$ an eigenvector of $\pi(-C_I)|_{V_\pi^{\HH_{P,k-1}}}$. 
Consequently, 
\begin{equation}\label{eq4:spec_I|C_Pk}
\Spec (\Delta_I|_{C_{P,k}^\infty(G)}) = \bigcup_{\pi\in\widehat G} \big\{\!\big\{ \rule{-4mm}{0mm}
\underbrace{\lambda^{\pi},\dots, \lambda^{\pi}}_{(d_\pi\dim V_\pi^{\HH_{P,k-1}})\text{-times}} 
\rule{-4mm}{0mm}\big\}\!\big\}.
\end{equation}
Recall from Remark~\ref{rem2:C_I} that $\lambda^\pi$ is determined by $\pi(-C_I) =  \lambda^\pi\, \Id_{V_\pi}$.
When $C_{P,k}^\infty(G)$ contains only constant functions on $G$ (e.g.\ if $\HH_{P,k-1}$ is bracket generating because $\dim V_\pi^{\HH_{P,k-1}}=0$ for all $\pi\in\widehat G$), we set $\lambda_1(\Delta_I|_{C_{P,k}^\infty(G)})=\infty$ by convention.

\begin{proposition}\label{prop4:lambda1-k}
Let $A\in \GL(m,\R)$.
For any $P$ sorting $A$ (see Notation~\ref{not2:sigma_j(A)}) and any index $1\leq k\leq m$, we have that
\begin{equation}\label{eq3:lambda_1(A)}
\lambda_1(G,\HH_{P,k},g_I|_{\HH_{P,k}}) \; \sigma_k(A)^2
\leq \lambda_1(G,g_A) \leq 
\lambda_1(\Delta_I|_{C_{P,k}^\infty(G)})\; \sigma_k(A)^2 . 
\end{equation}
\end{proposition}

\begin{proof}
Throughout the proof we abbreviate $\sigma_j=\sigma_j(A)$ for any $j$, i.e.\ $\sigma_1^2\geq \dots\geq \sigma_m^2$ are the eigenvalues of $AA^t$.
Fix $P\in \Ot(m)$ sorting $A$, that is, $AA^t=PD^2P^t$ with $D=\diag(\sigma_1,\dots,\sigma_m)$.
	
We have seen in the proof of Proposition~\ref{prop3:diam-k} that 
\begin{equation*}
\sigma_k^2 \, PD_1^2P^t\leq AA^t\leq \sigma_k^2 \, PD_2^2P^t
\end{equation*}
where 
$
D_1 =\diag( {1,\dots,1}, {\tfrac{\sigma_{k+1}}{\sigma_k},\dots, \tfrac{\sigma_m}{\sigma_k}})
$ and $
D_2 = \diag( {\tfrac{\sigma_1}{\sigma_k},\dots, \tfrac{\sigma_{k-1}}{\sigma_k}}, {1,\dots,1}).
$
It follows from Lemma~\ref{lem4:lambda1-inequality} that 
\begin{equation*}
	\sigma_k^2\, \lambda_1(G,g_{PD_1}) 
	= \lambda_1(G,g_{\sigma_k PD_1}) 
	\leq \lambda_1(G,g_A)\leq 
	\lambda_1(G,g_{\sigma_k PD_2}) =
	\sigma_k^2\, \lambda_1(G,g_{PD_2}). 
\end{equation*}
It remains to show that 
	\begin{align}\label{eq4:claim-prop-lambda_1k}
		\lambda_1(G,\HH_{P,k},g_I|_{\HH_{P,k}})
		\leq \lambda_1(G,g_{PD_1})
		\qquad\text{ and }\qquad
		\lambda_1(G,g_{PD_2})\leq \lambda_1(\Delta_I|_{C_{P,k}^\infty(G)}).
	\end{align}
	
From Lemma~\ref{lem2:C_AB}, we have that 
\begin{align*}
	C_{PD_1} = \sum_{i=1}^{k} X_i(P)^2 + \sum_{i=k+1}^m (\tfrac{\sigma_i}{\sigma_k})^2\, X_i(P)^2.
\end{align*}
Let $(\pi,V_\pi) \in\widehat G$ non-trivial. 
We abbreviate $C_{(P,k)}= C_{(\HH_{P,k},g_I|_{\HH_{P,k}})}$ (see Subsection~\ref{subsec:spec-non-Riemannian}).
Since $\pi(-X_i(P)^2):V_\pi\to V_\pi$ is positive semi-definite for every $i$, we obtain that 
\begin{align*}
\pi(-C_{(P,k)}) = \sum_{i=1}^{k} \pi(-X_i(P)^2) \leq \pi(-C_{PD_1}).
\end{align*}
Consequently, $\lambda_{\min}(\pi(-C_{(P,k)})) \leq \lambda_{\min}(\pi(-C_{PD_1}))$, thus the first inequality in \eqref{eq4:claim-prop-lambda_1k} follows by \eqref{eq2:lambda1(G,g_A)} and \eqref{eq2:lambda_1(G,H,g_I|H)}.

We now establish the inequality at the right-hand side in \eqref{eq4:claim-prop-lambda_1k}. 
We assume that the dimension of $C^\infty_{P,k}(G)$ is greater than one, otherwise the assertion follows trivially. 
From \eqref{eq4:spec_I|C_Pk}, it suffices to show that $\lambda_1(G,g_{PD_2})\leq \lambda^\pi$ for all $\pi\in\widehat G$ satisfying that $\dim V_\pi^{\HH_{P,k-1}}>0$. 
	
Let $\pi_0\in\widehat G$ satisfying $V_{\pi_0}^{\HH_{P,k-1}}\neq 0$ and let $v_0\in V_{\pi_0}^{\HH_{P,k-1}}$ with $\langle v_0, v_0\rangle_{\pi_0}=1$. 
By Lemma~\ref{lem2:C_AB}, 
\begin{align*}
	C_{PD_2} = \sum_{i=1}^{k-1} (\tfrac{\sigma_i}{\sigma_k})^2\,  X_i(P)^2 + \sum_{i=k}^m  X_i(P)^2.
\end{align*}
Note that $\pi_0(X_i(P))v_0=0$ for all $1\leq i\leq k-1$. 
Hence
\begin{multline*}
\lambda_1(G,g_{PD_2})
		\leq \; \lambda_{\min}(\pi_0(-C_{PD_2})) 
		= \min_{v\in V_{\pi_0}:\, \langle v,v\rangle_{\pi_0}=1} \langle \pi_0(-C_{PD_2})v,v \rangle_{\pi_0} \\
		\leq \langle \pi_0(-C_{PD_2})v_0,v_0\rangle_{\pi_0}
		=\sum_{j=k}^m \langle \pi_0(-X_j(P)^2)v_0,v_0\rangle_{\pi_0} \\
		\quad = \sum_{j=1}^m \langle \pi_0(-X_j(P)^2)v_0,v_0\rangle_{\pi_0} 
		=\langle \pi_0(-C_I)v_0,v_0\rangle_{\pi_0} = \lambda^{\pi_0},
\end{multline*}
and the proof is complete. 
\end{proof}

\begin{remark}
	One can check that $f\in C_{P,k}^\infty(G)$ if and only if it is annihilated by any element in the Lie subalgebra $\fh_{P,k-1}$ of $\fg$ generated by $\HH_{P,k-1}$, which is equivalent of being invariant by the closure of the only connected subgroup of $G$ with Lie algebra $\fh_{P,k-1}$.
In particular, every $f\in C_{P,\ell(P)}^\infty(G)$ is invariant by $\bar H_P$ (see Definition~\ref{def3:bracket-generating-indexA}), thus it induces a smooth function on the homogeneous space $G/\bar H_P$. 
Consequently, 
\begin{equation}
\lambda_1(\Delta_I|_{C_{P,k}^\infty(G)}) \leq \lambda_1(\Delta_I|_{C^\infty(G)^{\bar H_P}}) = \lambda_1(G/\bar H_P,g_I|_{\bar\fh_P^\perp}),
\end{equation} 
and $\lambda_1(\Delta_I|_{C_{P,k}^\infty(G)})<\infty$ if and only if $\bar H_P\neq G$ by Remark~\ref{rem2:G/H}. 
\end{remark}

\subsection{Optimal indices for the first eigenvalue} \label{subsec:lambda1-optimal}
We next obtain the analogous results as in Subsection~\ref{subsec:diam-optimal}.

\begin{theorem}\label{thm4:optimal-index-upper-bound}
	Assume that $G$ is non-abelian. 
	There is a positive real number $C_4$ depending on $G$ and $g_I$ such that 
	\begin{equation}\label{eq4:optimalupper}
		\lambda_1(G,g_A) \leq 
		C_4\, \sigma_{2}(A)^2
\qquad\text{for all $A\in\GL(m,\R)$.}
	\end{equation}
Moreover, 
	\begin{equation}
		\label{eq4:lambda1-optimal-non-true-upperbound}
		\sup_{A\in\GL(m,\R) } \frac{\lambda_1(G,g_A)} {\sigma_{4}(A)^2} =\infty.
	\end{equation}
	Consequently, the largest index $k$ satisfying that $\sup_{A\in\GL(m,\R)} \lambda_1(G,g_A)/ \sigma_{k}(A)^2<\infty$ is $2$ or $3$. 
\end{theorem}

\begin{proof}
	We first note that \eqref{eq4:lambda1-optimal-non-true-upperbound} follows immediately from 
	\eqref{eq4:PeterLi} and \eqref{eq3:diam-optimal-non-true-lowerbound}.

	Let $T$ be a maximal torus in $G$ with Lie algebra $\ft$. 
	We have that $\ft \neq \fg$ by assumption. 
	We claim that \eqref{eq4:optimalupper} holds by setting $C_4= \lambda_1(G/T, g_I|_{\ft^{\perp}})$. 
	
	We fix $A\in \GL(m,\R)$.
	Let $P$ be any matrix in $\Ot(m)$ sorting $A$.
	Write $\mathcal C=\mathcal C_{P,2} = \Span_\R\{ X_2(P),\dots,X_m(P) \}$.
	By Proposition~\ref{prop4:lambda1-k}, we have that 
	$$
	\lambda_1(G,g_A)\leq \lambda_1(\Delta_I|_{C_{P,2}^\infty(G)}) \;\sigma_2(A)^2.
	$$

	The subspace $\Span_\R\{X_1(P)\}$ of $\fg$ is of course an abelian subalgebra of $\fg$.
	Since any two maximal abelian subalgebras of $\fg$ are conjugate via $\Ad(G)$ (see for instance \cite[Thm.~4.34]{Knapp-book-beyond}), there is $a\in G$ such that $X_1(P) \in \Ad_a(\ft)$. 
	Then
	\begin{align*}
		\lambda_1(\Delta_I|_{C_{P,2}^\infty(G)}) \leq \lambda_1(\Delta_I|_{C^\infty(G)^{aTa^{-1}}}) = \lambda_1(\Delta_I|_{C^\infty(G)^{T}}) 
		= \lambda_1(G/T,g_I|_{\ft^\perp}),
	\end{align*}
	which completes the proof.
\end{proof}

\begin{remark}\label{rem4:flat-upper}
	When $G$ is abelian, one has that $\sup_{A\in\GL(m,\R) } {\lambda_1(G,g_A)}\, {\sigma_{2}(A)^{-2}} =\infty$ by using the same construction as in Remark~\ref{rem3:flat-lower}. 
	Consequently, the index $k=1$ at the right-hand side of \eqref{eq4:lambda1-simple-estimates} is optimal. 
\end{remark}

We recall from \eqref{eq3:k_max} that $k_{\max}= 1+\max_{H} \,\dim H$, where $H$ runs over the closed subgroups of $G$ with Lie algebra $\fh\neq\fg$.
One has that $k_{\max}=m$ if and only if $G$ is not semisimple.

\begin{theorem}\label{thm4:optimal-index-lower-bound}
We have that
\begin{equation}
	\label{eq4:lambda1-optimal-non-true-lowerbound}
	\inf_{A\in\GL(m,\R) } \frac{\lambda_1(G,g_A)} {\sigma_{k_{\max}-1}(A)^2} =0.
\end{equation}
Consequently, if $G$ is not semisimple, then $k_{\max}=m$ is the smallest index $k$ satisfying that $\inf_{A\in\GL(m,\R) }\lambda_1(G,g_A)/ \sigma_{k}(A)^2>0$. 

When $G$ is semisimple, if Condition~\ref{claim} holds\footnote{See the footnote in page 17 for an update.}, then there is a positive real number $C_3$ depending on $G$ and $g_I$ such that 
\begin{equation}\label{eq4:optimallower}
\lambda_1(G,g_A)\geq C_3\, \sigma_{k_{\max}}(A)^2 
\qquad\text{for all $A\in\GL(m,\R)$}.
\end{equation}
Consequently, $k_{\max}$ is the smallest index $k$ satisfying that $\inf_{A\in\GL(m,\R) }\lambda_1(G,g_A)/ \sigma_{k}(A)^2>0$. 

\end{theorem}

\begin{proof}
We first prove \eqref{eq4:lambda1-optimal-non-true-lowerbound}, in a very similar way as \eqref{eq3:diam-optimal-non-true-upperbound}.
Let $H$ be any proper closed subgroup of $G$ with Lie algebra $\fh\neq\fg$. 
Write $n=\dim(H)$. 
Clearly, $n\leq k_{\max}-1$ by \eqref{eq3:k_max}.
There is $P\in\Ot(m)$ such that $\fh = \Span_\R\{X_1(P),\dots, X_n(P)\}$. 
For $s>0$, let 
$$
D_s=\diag(\underbrace{1,\dots,1}_{n\text{-times}}, \underbrace{s,\dots,s}_{(m-n)\text{-times}}).
$$
One clearly has that $\sigma_n(PD_s)=1$ for all $s\leq 1$. 
We claim that 
$$
\lim_{s\to0^+} \;\frac{\lambda_1(G,g_{PD_s})}{\sigma_n(PD_s)^2}=
\lim_{s\to0^+} \;\lambda_1(G,g_{PD_s}) =0.
$$ 
Let $\pi_0$ be any irreducible representation of $G$ satisfying that $V_{\pi_0}^H\neq 0$.
We have that $\pi_0(X) v=0$ for all $v\in V_{\pi_0}^H$ and $X\in\fh$. 
Hence, if $v_0\in V_{\pi_0}^H$ with $v_0\neq 0$, then 
\begin{align*}
	\lambda_1(G,g_{PD_s})
	&\leq \; \lambda_{\min}(\pi_0(-C_{PD_s})) 
	= \min_{v\in V_{\pi_0}:\, \langle v,v\rangle_{\pi_0}=1} \langle \pi_0(-C_{PD_s})v,v \rangle_{\pi_0} \\
	&\leq \langle \pi_0(-C_{PD_s})v_0,v_0\rangle_{\pi_0}
	= s^2\sum_{j=n+1}^m \langle \pi_0(-X_j(P)^2)v_0,v_0\rangle_{\pi_0}.
\end{align*}
The proof of \eqref{eq4:lambda1-optimal-non-true-lowerbound} follows by taking $H$ of dimension $n=k_{\max}-1$. 

If $G$ is not semisimple, then $k_{\max}=m$, thus the right hand side of \eqref{eq4:lambda1-simple-estimates} gives already the optimal index. 

Peter Li's estimate \eqref{eq1:Li-estimate} restricted to left-invariant metrics on $G$ gives 
\begin{equation}\label{eq4:PeterLi}
	\frac{\pi^2}4\leq \lambda_1(G,g_A)\diam(G,g_A)^2 
\qquad\text{for all $A\in\GL(m,\R)$.}
\end{equation}
If $G$ is semisimple and Condition~\ref{claim} holds, then  $\diam(G,g_A)\leq \sigma_{k_{\max}}(A)^{-1} C_{2}$ for all $A\in \GL(m,\R)$ by Theorem~\ref{thm3:optimal-index-upper-bound}, thus \eqref{eq4:PeterLi} implies that  \eqref{eq4:optimallower} holds with $C_3=\frac{\pi^2}{4C_2^2}$.
\end{proof}

\begin{remark}\label{rem4:SU(2)}
For $G=\SU(2)$ and $\SO(3)$, as a direct consequence of an explicit expression of $\lambda_1(G,g_A)$ for any $A\in\GL(3,\R)$, \eqref{eq4:optimallower} and \eqref{eq4:optimalupper} were established  in \cite[Cor.~4.5]{Lauret-SpecSU(2)} as follows:
\begin{align*}
2\,\sigma_2(A)^2 &<\lambda_1(\SU(2),g_A)\leq 8\,\sigma_2(A)^2, 
&
4\,\sigma_2(A)^2 &<\lambda_1(\SO(3),g_A)\leq 8\,\sigma_2(A)^2.
\end{align*}
Moreover, the upper bounds are attained on $g_A$ satisfying that $\sigma_2(A)=\sigma_3(A)$ and $\sigma_1(A)$ is large enough, and the lower bounds are asympotically sharp approached by $g_A$ with $\sigma_1(A)=\sigma_2(A)$ and $\sigma_3(A)\to0$.
\end{remark}

\subsection{First eigenvalue estimate for a restricted subclass}

Analogously to Subsection~\ref{subsec:diam-restricted}, we next look for better indices by restricting the set of metrics to consider. 
The objects introduced in Definition~\ref{def3:bracket-generating-indexA} will be used here; in particular, for $P\in\Ot(m)$ fixed, the subclass of left-invariant metrics $\mathcal M^G(P)$ on $G$ given by the elements $g_{PQD}$ for $D=\diag(\sigma_1,\dots,\sigma_m)$ with $ \sigma_1\geq\dots\geq \sigma_m>0$ and 
\begin{equation}
Q=
\left(\begin{smallmatrix}
Q_1\\ &1\\ &&Q_2
\end{smallmatrix}\right)\in\Ot(m)
\end{equation}
with $Q_1\in\Ot(k-1)$ and $Q_2\in \Ot(m-k)$. 

\begin{theorem}\label{thm4:lambda1(g_PQD)}
Let $P\in\Ot(m)$ and set $k=\ell(P)$. 
We have that 
\begin{equation}\label{eq4:lambda1(g_PQD)geq}
\lambda_1(G,g_{A}) \geq 
\lambda_1(G,\HH_{P,k}, g_I|_{\HH_{P,k}})\;  \sigma_{k}(A)^2 >0
\qquad \text{for every } g_A\in \mathcal M^G(P).
\end{equation}
Furthermore, if $\bar\HH_P\neq G$ 
(e.g.\ if $G$ is semisimple), then 
\begin{equation}\label{eq4:lambda1(g_PQD)leq}
\lambda_1(G,g_{A}) \leq \lambda_1(G/\bar H_P,g_I|_{\bar\fh_P^\perp}) \; \sigma_{k}(A)^2 <\infty
\qquad \text{for every }g_A\in \mathcal M^G(P).
\end{equation}
\end{theorem}

\begin{proof}
For $Q$ as above (i.e.\ $Q\in\mathcal O(m,k)$), we proved in the proof of Theorem~\ref{thm3:diam(g_PQD)} that $\HH_{PQ,k}= \HH_{P,k}$, $\CC_{PQ,k}=\CC_{P,k}$, and $\ell(PQ)=\ell(P)=k$. 
Hence, for any $D$ as above (i.e.\ $D\in\mathcal D(m)$), Proposition~\ref{prop4:lambda1-k} implies that 
\begin{equation*}
\lambda_1(G,\HH_{P,k},g_I|_{\HH_{P,k}} )\; \sigma_k(PQD)^2
\leq \lambda_1(G,g_{PQD}) \leq 
\lambda_1(\Delta_I|_{C_{P,k}^\infty(G)})\; \sigma_k(PQD)^2.
\end{equation*}
From Lemma~\ref{lem2:sub-spec-discrete}, it follows that $\lambda_1(G,\HH_{P,k},g_I|_{\HH_{P,k}}) >0$, showing \eqref{eq4:lambda1(g_PQD)geq}. 
	
Since $\fh_P\subset \bar \fh_P \neq \fg$, we have that $C_{P,k}^\infty(P)\supset C^\infty(G)^{\bar H_P} \equiv C^\infty(G/\bar H_P)$. 
Therefore 
$$
\lambda_1(\Delta_I|_{C_{P,k}^\infty(G)})\leq \lambda_1(\Delta_I|_{C^\infty(G)^{\bar H_P}}) =  \lambda_1(G/\bar H_P,g_I|_{\bar\fh_P^\perp}).
$$
That $\lambda_1(G/\bar H_P,g_I|_{\bar\fh_P^\perp})<\infty$ follows from $\bar H_P\neq G$, and \eqref{eq4:lambda1(g_PQD)leq} is proved. 
\end{proof}

We now replace $\lambda_1(G/\bar H_P,g_I|_{\bar \fh_P^\perp})$ in \eqref{eq4:lambda1(g_PQD)leq} by a constant independent from $P$ when $G$ is semisimple.

\begin{corollary}\label{cor4:lambda1-inf-universal}
Let $G$ be a compact connected semisimple Lie group. 
Under the notation introduced in Remark~\ref{rem3:max-subalgebras}, we set
\begin{equation}
C_4'=\max_{1\leq i\leq r} \, \lambda_1(G/H_i,g_I|_{\fh_i^{\perp}}),
\end{equation}
which is positive and depends only on $G$ and $g_I$. 
Then, for any $P\in\Ot(m)$, 
\begin{equation}\label{eq4:lambda1P-geq-independent}
\lambda_1(G,g_{A}) \leq C_4'\;\sigma_{\ell(P)}(A)^2
\qquad\text{for every } A\in\mathcal M^G(P).
\end{equation}
\end{corollary}

\begin{proof}
Fix any $P\in\Ot(m)$. 
Since $G$ is semisimple, $\bar\fh_P\neq\fg$, thus $\bar\fh_P\subset \Ad_a(\fh_i)$ for some $i$ and $a\in G$ (see Remark~\ref{rem3:max-subalgebras}). 
We note that $aHa^{-1}$ is the connected subgroup of $G$ with Lie algebra $\Ad_a(\fh_i)$. 
It follows immediately that
\begin{equation*}
\lambda_1(G/\bar H_P,g_I|_{\bar\fh_P^\perp}) \leq 
\lambda_1(G/aH_ia^{-1}, g_I|_{\Ad_a(\fh_i^\perp)}) 
= \lambda_1(G/H_i, g_I|_{\fh_i^\perp})\leq C_4',
\end{equation*}
and the proof is complete by \eqref{eq4:lambda1(g_PQD)leq}. 
\end{proof}

\begin{remark}
There should not exist a lower bound for $\lambda_1(G,g_A)$ for all $g_A\in\mathcal M^G(P)$ independent on $P$ analogous to \eqref{eq4:lambda1P-geq-independent}. 
This is because $\lambda_1(G,\HH_{P,\ell(P)},g_I|_{\HH_{P,\ell(P)}})$ may not be bounded by above uniformly for all $P\in\Ot(m)$. 
\end{remark}

\section{On the EGS conjecture} \label{sec:upperbounds}

In this section we combine the diameter estimates from Section~\ref{sec:diam} and the eigenvalue estimates from Section~\ref{sec:eigenvalues} to give partial answers to the EGS conjecture (Conjecture~\ref{conj:EGS}). 
We still consider $G$ a compact connected Lie group with a fixed $\Ad(G)$-invariant inner product $g_I$ with  orthonormal basis $\mathcal B=\{X_1,\dots,X_m\}$. 
In addition, we will assume that $G$ is non-abelian.

Recall that $\sigma_1(A)\geq\dots\geq \sigma_m(A)>0$ were introduced in Notation~\ref{not2:sigma_j(A)} for $A\in\GL(m,\R)$. 
Theorems~\ref{thm3:optimal-index-upper-bound} and \ref{thm4:optimal-index-upper-bound} imply that
\begin{equation}\label{eq5:optiaml}
\lambda_1(G,g_A)\, \diam(G,g_A)^2 \leq C_2^2\, C_4 \left( \frac{\sigma_{2}(A)} {\sigma_{k_{\max}}(A)} \right)^2
\qquad\text{for every $A\in\GL(m,\R)$}.
\end{equation}
(When $G$ is semisimple, this holds provided\footnote{See the footnote in page 17 for an update.} Condition~\ref{claim} is true.)
This is far from being sufficient to establish the EGS conjecture for a general $G$ since the term $\sigma_2(A)/\sigma_{k_{\max}}(A)$ is of course not bounded by above when $k_{\max}>2$. 
For instance, the following table shows $k_{\max}$ in some standard cases: 
\begin{equation}
\begin{array}{cccc}
G  & k_{\max} &\dim G\\ \hline 
\rule{0pt}{12pt}
\SU(n) & n^2-2n+2&n^2-1\\
\Ut(n) & n^2 & n^2 \\
\SO(n) & \tfrac12 (n^2-3n+4)& \tfrac12 n(n-1)\\
\Sp(n) & 2n^2-3n+4&n(2n+1)\\
\SU(2)\times\SU(2)  & 5 & 6 
\end{array}
\end{equation}
In fact, the next result tells us that \eqref{eq5:optiaml} establishes the EGS conjecture only for those $G$ that is already known. 

\begin{proposition}
For $G$ a non-abelian compact Lie group, $k_{\max}(G)=2$ if and only if $\fg\simeq\su(2)$. 
\end{proposition}

\begin{proof}
The converse is clear. 
A standard fact obtained by using the root system associated to $\fg_\C$ is that there always exists a subalgebra of $\fg$ isomorphic to $\su(2)$. 
Hence, $k_{\max}\geq \dim\su(2)=3>2$ unless $\fg\simeq \su(2)$, which concludes the proof.  
\end{proof}

We are now in position to give a partial answer to Conjecture~\ref{conj:EGS}. 

\begin{theorem}\label{thm5:EGSweak}
Let $G$ be a compact connected Lie group of dimension $m$. 
Let $\mathcal S$ be a finite union of the following sets:
\begin{enumerate} 
\item \label{item:sigma_1/sigma_m}
$\Sigma(c_0):= \{ g_A: A\in \GL(m,\R)\text{ and } 		\sigma_2(A)\leq c_0\, \sigma_{k_{\max}}(A) \}$ for any $c_0\geq1$ fixed. 

\item \label{item:P} \rule{0pt}{14pt}
$\mathcal M^G(P)$ for any $P\in\Ot(m)$ satisfying $\bar\HH_P\neq G$ (e.g.\ when $G$ is semisimple).
\end{enumerate}
Then, there exists a positive real number $C$ depending on $(G,g_I,\mathcal B, \mathcal S)$ such that 
\begin{equation}\label{eq:lambda_1diam^2leqc_1}
\lambda_1(G,g)\, \diam(G,g)^2\leq C
\qquad\text{for all $g\in \mathcal S$. }
\end{equation}
\end{theorem}

\begin{proof}
Clearly, it is sufficient to consider the case when $\mathcal S$ is a single set as in \eqref{item:sigma_1/sigma_m} or \eqref{item:P}. 
By \eqref{eq5:optiaml}, the case $\mathcal S=\Sigma(c_0)$ for some $c_0>0$ follows immediately with $C=C_2^2C_4 c_0^2$.

We now assume that $\mathcal S=\mathcal M^G(P)$ for some $P\in\Ot(m)$. 
Set $k=\ell(P)$, which depends on $P$ and also on $\mathcal B$. 
For any $g_A\in\mathcal M^G(P)$, Theorems~\ref{thm3:diam(g_PQD)} and \ref{thm4:lambda1(g_PQD)} yield
\begin{equation*}
\lambda_1(G,g_{A})\, \diam(G,g_{A})^2 \leq 
\lambda_1(G/\bar H_P,g_I|_{\bar\fh_P^\perp})\,
{\diam(G,\HH_{P,k}, g_I|_{\HH_{P,k}})^2} . 
\end{equation*}
The assertion follows since the term at the right-hand side depends only on $G$, $g_I$, $\mathcal B$, and $P$. 
\end{proof}

Theorem~\ref{thm:main} is a particular case of Theorem~\ref{thm5:EGSweak}.
In fact, the set $\mathcal M^G(\fa,Y)$ in Theorem~\ref{thm:main} coincides with $\mathcal M^G(P)$ for any $P\in\Ot(m)$ satisfying that $Y \in \Span_\R \{X_{\ell(P)}\}$ and $\fa=\Span_\R\{X_1(P),\dots,X_{\ell(P)-1}(P)\}$.

\begin{proof}[Proof of Corollary~\ref{cor:main}]
There is $P\in\Ot(m)$ such that $Y_j=X_j(P)$ for all $1\leq j\leq m$. 
Let $\mathbb S$ denote the set of permutation $m\times m$ matrices.
Of course, $\mathbb S$ is contained in $\Ot(m)$ and it has $m!$ elements. 
One has that
\begin{align}
\mathcal M^G(\mathcal B) 
&= \{g\in\mathcal M^G: g(Y_i,Y_j)=0 \text{ for all } i\neq j \}
\\ \notag
&= \{g_{PD}: D=\diag(d_1,\dots,d_m),\; d_1,\dots,d_m>0\}
\\ \notag
&= \{g_{PRD}: D\in\mathcal D(m),\, R\in\mathbb S\}
\\ \notag 
&\subset \bigcup_{R\in \mathbb S} \mathcal M^G(PR). 
\end{align}
Hence, the assertion follows immediately by Theorem~\ref{thm5:EGSweak} since $\mathbb S$ is finite.
\end{proof}

\begin{remark}\label{rem5:dimension}
Note that every left-invariant metric is in $\mathcal M^G(P)$ for some $P\in\Ot(m)$. 
In fact, for any $g_A\in\mathcal M^G$ with $A\in\GL(m,\R)$,  $g_A\in \mathcal M^G(P)$ for any $P\in\Ot(m)$ sorting $A$.

We next show that the order of $\dim \mathcal M^G(P)$ increase as $\dim \mathcal M^G$ when $m=\dim G$ grows. 
We clearly have that $\dim\mathcal M^G =\dim\{\text{$m\times m$ positive definite matrices}\}= \frac12 m(m+1)$. 
We claim that 
\begin{equation}\label{eq:dimM(P)}
\begin{aligned}
\dim \mathcal M^G(P) &= \dim \mathcal D(m)+\dim \mathcal O(m,k)= m+\dim\Ot(k-1)+\dim\Ot(m-k) 
\\
&= m+\tfrac12 (k-1)(k-2) + \tfrac12 (m-k)(m-k-1).
\end{aligned}
\end{equation}
Recall that $g_A=g_B$ if and only if $AR=B$ for some $R\in\Ot(m)$. 
We now assume that two elements in $\mathcal M^G(P)$ coincide, say $g_{PQD}$ and $g_{PQ'D'}$ for some $Q,Q'\in\mathcal O(m,k)$ and $D,D'\in\mathcal D(m)$. Thus $PQD=PQ'D'R$ for some $R\in\Ot(m)$. 
It follows that $D=D'$, then $D=Q^tQ'DR$. 
By considering $D$ in the subspace of $\mathcal D(m)$ given by matrices with simple spectrum (i.e.\ all the diagonal entries are different pairwise), which has dimension $m=\dim \mathcal D(m)$, we obtain that $Q=Q'$ and $R=I$, and the claim follows. 

We now suppose that $P\in\Ot(m)$ satisfies $\ell(P)=2$, i.e.\ $\{X_1(P), X_2(P)\}$ is bracket-generating. 
Such element always exists if $G$ is semisimple. 
Then, \eqref{eq:dimM(P)} gives $\dim \mathcal M^G(P) = m+ \tfrac12 (m-2)(m-3)$ and $\dim \mathcal M^G-\dim \mathcal M^G(P)= 2m-3$.

Of course, we are not able to compute the dimensions of the isometry classes in $\mathcal M^G(P)$ since it is not know in general for $\mathcal M^G$. 
\end{remark}

We conclude the article by including some incomplete ideas of how to construct a counterexample or a proof-by-contradiction of Conjecture~\ref{conj:EGS}. 
For $A\in\GL(m,\R)$, we abbreviate $\dist_A(a,b)$ the distance between $a,b\in G$ with respect to the Riemannian metric $g_A$. 

\begin{proposition}\label{prop6:false} 
Assume that Conjecture~\ref{conj:EGS} is false, that is, there is a compact connected Lie group $G$ such that 
\begin{equation}\label{eq6:false}
\sup_{g\in \mathcal M^G} \lambda_1(G,g)\, \diam(G,g)^2=\infty. 
\end{equation}
Then, there are sequences $\{P_n\}_{n\in\N}\subset \Ot(m)$, $\{D_n\}_{n\in\N}\subset \mathcal D(m)$, $\{a_n\}_{n\in\N}\subset G$, and elements $P_0\in\Ot(m)$, $k\in\Z$, and $a_0\in G$, satisfying the following properties: 
\begin{itemize}
\item[(I)] $\displaystyle \lambda_1(G,g_{P_nD_n})\, \diam(G,g_{P_nD_n})^2\geq n$ for all $n\in\N$;
	
\item[(II)] $\displaystyle \lim_{n\to\infty} P_n = P_0$;

\item[(III)] $\ell(P_n)=k$ for all $n\in\N$;
	
\item[(IV)] $\displaystyle \diam(G,g_{P_nD_n}) = \dist_{P_nD_n}(e,a_n)$  for all $n\in\N$;
	
\item[(V)] $\displaystyle \lim_{n\to\infty} a_n = a_0$. 
\end{itemize}
\end{proposition}

\begin{proof}
It follows from \eqref{eq6:false} that there exists a sequence $\{A_n^{(1)}\}_{n\in \N} \subset \GL(m,\R)$ such that 
\begin{equation}\label{eq6:A_n^1}
\lambda_1(G,g_{A_n^{(1)}})\, \diam(G,g_{A_n^{(1)}})^2\geq n
\qquad\text{for all $n\in\N$. }
\end{equation}
For each $n\in\N$, let $P_n^{(1)}\in\Ot(m)$ and $D_n^{(1)}\in\mathcal D(m)$ such that $A_n^{(1)} (A_n^{(1)})^t=P_n^{(1)} (D_n^{(1)})^2 (P_n^{(1)})^t$, thus $g_{A_n^{(1)}} = g_{P_n^{(1)}D_n^{(1)}}$.
Since $\Ot(m)$ is compact, there is a subsequence  $\{A_n^{(2)}\}_{n\in\N}$ of $\{A_n^{(1)}\}_{n\in\N}$ satisfying that, for $P_n^{(2)}\in\Ot(m)$, $D_n^{(1)}\in\mathcal D(m)$ with $A_n^{(2)} (A_n^{(2)})^t=P_n^{(2)} (D_n^{(2)})^2 (P_n^{(2)})^t$, the sequence $\{P_n^{(2)}\}_{n\in\N}$ converges to some $P_0^{(2)}\in\Ot(m)$. 
Clearly, \eqref{eq6:A_n^1} still holds with $A_n^{(1)}$ replaced by $A_n^{(2)}$.

Since $\ell(P_n^{(2)})$ lies in the finite set $\{1,\dots,m\}$, at least one index $k$ is repeated infinitely many times.
Hence, we can assume, by taking a new subsequence, that $\ell(P_n^{(2)})$ is constant for all $n$. 

For each $n\in\N$, let $a_n^{(2)}$ be any element in $G$ satisfying that $\diam(G,g_{A_n^{(2)}}) = \dist_{{A_{n}^{(2)}}} (e,a_n^{(2)})$.
Since $G$ is compact, there is a new subsequence $\{A_n^{(3)}\}_{n\in\N}$ of $\{A_n^{(2)}\}_{n\in\N}$, with corresponding subsequences $\{P_n^{(3)}\}_{n\in\N}$ of $\{P_n^{(2)}\}_{n\in\N}$, $\{D_n^{(3)}\}_{n\in\N}$ of $\{D_n^{(2)}\}_{n\in\N}$, and $\{a_n^{(3)}\}_{n\in\N}$ of $\{a_n^{(2)}\}_{n\in\N}$, satisfying in addition that $a_n^{(3)}$ converges to some $a_0\in G$. 
\end{proof}

Note that, in the situation of Proposition~\ref{prop6:false}, Theorem~\ref{thm5:EGSweak} implies that 
\begin{equation}
\lim_{n\to\infty} {\sigma_2(D_n)}\, {\sigma_{k_{\max}}(D_n)^{-1}} =\infty. 
\end{equation}

One may think that, since $P_n$ is very close to $P_0$ for $n$ large, then 
\begin{equation}
|\lambda_1(G,g_{P_nD_n})\, \diam(G,g_{P_nD_n})^2 - \lambda_1(G,g_{P_0D_n})\, \diam(G,g_{P_0D_n})^2|
\end{equation}
will be small, or at least bounded by above. 
If this is true, then 
\begin{align}
\lim_{n\to\infty} \lambda_1(G,g_{P_0D_n})\, \diam(G,g_{P_0D_n})^2 
=\infty,
\end{align}
which contradicts Theorem~\ref{thm5:EGSweak}.

\bibliographystyle{plain}

\end{document}